\documentclass[11pt]{amsart}
\usepackage[letterpaper,margin=1.1in]{geometry}
\usepackage{amsmath}
\usepackage{amsthm}
\usepackage{amssymb}
\usepackage{amsfonts}
\usepackage{graphicx}
\usepackage{mathrsfs,bm,amscd}
\usepackage{latexsym}
\usepackage{color}
\usepackage{marginnote}
\usepackage{fourier}

\usepackage[all]{xy}

\numberwithin{equation}{section}
\theoremstyle{plain}
\newtheorem{thm}{Theorem}[section]
\newtheorem{prop}[thm]{Proposition}
\newtheorem{cor}[thm]{Corollary}
\newtheorem{lem}[thm]{Lemma}

\theoremstyle{definition}
\newtheorem{defn}{Definition}[section]

\theoremstyle{remark}
\newtheorem{rem}[thm]{Remark}
\newtheorem{ex}[thm]{Example}
\newtheorem*{bprob*}{Bonus problem}

\def\R{\mathbb{R}}
\def\Z{\mathbb{Z}}
\def\N{\mathbb{N}}

\def\S{\mathbb{S}}

\def\H{\mathbb{H}}
\def\e{\epsilon}
\def\g{\gamma}
\def\d{\delta}

\DeclareMathOperator{\diam}{diam}
\DeclareMathOperator{\Lip}{Lip}

\DeclareMathOperator{\dist}{dist}

\newcommand{\bH}{\mathbb{H}}

\begin{document}

\title[Bi-Lipschitz embeddings of Heisenberg submanifolds]{Bi-Lipschitz embeddings of Heisenberg submanifolds into Euclidean spaces}

\begin{abstract} The Heisenberg group $\bH$ equipped with a sub-Riemannian metric is one of the most well known examples of a doubling metric space which does not admit a bi-Lipschitz embedding into any Euclidean space. In this paper we investigate which \textit{subsets} of $\bH$ bi-Lipschitz embed into Euclidean spaces. We show that there exists a universal constant $L>0$ such that lines $L$-bi-Lipschitz embed into $\R^3$ and planes $L$-bi-Lipschitz embed into $\R^4$.
Moreover, $C^{1,1}$ $2$-manifolds without characteristic points as well as all $C^{1,1}$ $1$-manifolds locally $L$-bi-Lipschitz embed into $\R^4$ where the constant $L$ is again universal. We also consider several examples of compact surfaces with characteristic points and we prove, for example, that Kor\'{a}nyi spheres bi-Lipschitz embed into  $\R^4$ with a uniform constant. Finally, we show that there exists a compact, porous subset of $\bH$ which does not admit a bi-Lipschitz embedding into any Euclidean space.

\end{abstract}

\author{Vasileios Chousionis \and Sean Li \and Vyron Vellis \and Scott Zimmerman}
\thanks{V.C. is supported by  the Simons Collaboration grant no.\  521845. S.L. is supported by the NSF grant DMS-1600804. V.V. is supported by the NSF grant DMS-1800731.}
\date{\today}
\subjclass[2010]{Primary 30L05; Secondary 53C17}
\keywords{Heisenberg group, bi-Lipschitz embedding, porous sets}

\address{Department of Mathematics\\ University of Connecticut\\ Storrs, CT 06269-1009}
\email{vasileios.chousionis@uconn.edu}
\address{Department of Mathematics\\ University of Connecticut\\ Storrs, CT 06269-1009}
\email{sean.li@uconn.edu}
\address{Department of Mathematics\\ University of Connecticut\\ Storrs, CT 06269-1009}
\email{vyron.vellis@uconn.edu}
\address{Department of Mathematics\\ University of Connecticut\\ Storrs, CT 06269-1009}
\email{scott.zimmerman@uconn.edu}

\maketitle

\section{Introduction}

Given two metric spaces $(X,d_X)$ and $(Y,d_Y)$, we say that a mapping $f:X\to Y$ is $L$\emph{-bi-Lipschitz} for $L\geq 1$ if 
$$
L^{-1} d_X(a,b) \leq d_Y(f(a),f(b)) \leq Ld_X(a,b) \quad \text{for any } a,b \in X.
$$
In this case we say that $X$ \emph{bi-Lipschitz embeds in $Y$} or that $X$ \emph{admits a bi-Lipschitz embedding into} $Y$. For two metric spaces $X$ and $Y$, we define the quantity
$$
\Lip(X,Y) := \inf \{L : \text{ there exists $L$-bi-Lipschitz } f : X \to Y\} \in [1,\infty]
$$
so that $X$ admits a bi-Lipschitz embedding into $Y$ if and only if $\Lip(X,Y) < \infty$.

When does a metric space admit a bi-Lipschitz embedding into some (finite-dimensional) Euclidean space? Spaces admitting such embeddings can roughly be thought to live inside a Euclidean space, and this is very important for a comprehensive study of their geometry. The embedding problem has attracted considerable attention over the years due to its applications in theoretical computer science and, more specifically, in graphic imaging and storage and access issues for large data sets \cite{KV, Naor}.

It is well know that a metric space bi-Lipschitz embeds into some Euclidean space only if it is doubling. Recall that a space is \emph{doubling} if every ball of radius $r$ can be covered by at most $N$ balls of radii $r/2$ for some fixed $N>1$.  Moreover Assouad \cite{As1, As2, As3} showed that, if a metric space $(X,d)$ is doubling, then, for any $\e > 0$, the snowflaked space $(X,d^{\e})$ bi-Lipschitz embeds into some Euclidean space. Nevertheless, the doubling condition is not sufficient for the existence of a bi-Lipschitz embedding into a Euclidean space. The Heisenberg group endowed with a sub-Riemannian metric is probably the most well known example of a doubling space which does not admit bi-Lipschitz embedding into any Euclidean space. This follows from a  deep theorem of Pansu \cite{Pansu} and an observation of Semmes \cite{Semmes3}. In fact, the second author showed that $\bH$ does not bi-Lipschitz embed in any Hilbert space \cite{li}. We also record that bi-Lipschitz embeddability of sub-Riemannian manifolds, and especially of Carnot groups, has been studied by various authors \cite{Semmes3,Wu,RV,Wu2,Romney}.

%For related results see Semmes \cite{Semmes99}.

%A necessary and sufficient condition for global bi-Lipschitz embedabbility in terms of local embedabbility was given by  Seo \cite{Seo}. Other sufficient conditions have been given by Eriksson-Bique \cite{E-B} in terms of bounds of sectional curvature and by Romney \cite{Romney} in terms of bounds of Alexandroff curvature.(\textcolor{red}{Why do we have this  paragraph?})

%The bi-Lipschitz embedabbility of sub-Riemannian manifolds has been studied by various authors \cite{Semmes3,Wu,RV,Wu2,Romney}. Carnot groups comprise a large class of canonical sub-Riemannian manifolds. The \emph{Heisenberg group} $\mathbb{H}$ is the simplest Carnot group that is not Euclidean; see \textsection\ref{sec:prelim} for definitions. Yet, it follows from a  deep theorem of Pansu \cite{Pansu} and an observation of Semmes \cite{Semmes3} that $\mathbb{H}$ does not embed in a bi-Lipschitz way into any Euclidean space. In fact, the second author showed that $\bH$ does not bi-Lipschitz embed in any Hilbert space \cite{li}.

Here, we are concerned with the following question: which subsets of $\mathbb{H}$ admit a bi-Lipschitz embedding into some Euclidean space? Clearly, every finite subset of the Heisenberg group admits such an embedding. On the other hand, following the arguments of Pansu and Semmes, it is easy to see that no open subset of the Heisenberg group admits such a bi-Lipschitz embedding. We are particularly interested in the bi-Lipschitz embeddability of submanifolds of $\bH$, which have been studied extensively in the past 20 years in connection to geometric measure theory and geometric analysis, see e.g. \cite{sc,CDPT}.

Our first theorem focuses on affine subsets of $\R^3$.
Here and in what follows, $d_K$ is the Kor\'{a}nyi metric defined in Section~\ref{sec:prelim}.

\begin{thm}\label{thm:line-plane}
There exists a universal $L>1$ with the following properties.
\begin{enumerate}
\item If $\ell \subset \R^3$ is a line, then $\Lip((\ell, d_K),\R^3) \leq L$.
\item If $P \subset \R^3$ is a plane, then $\Lip((P, d_K),\R^4) \leq L$.
\end{enumerate}
\end{thm}

We note that the Euclidean dimensions in the preceding theorem are sharp. Indeed, the $z$-axis in the Heisenberg group is bi-Lipschitz homeomorphic to the snowflaked space $(\R,|\cdot|^{1/2})$ which cannot be embedded in $\R^2$. On the other hand, the plane $\{x=0\}$ is bi-Lipschitz homeomorphic to the product space $(\R,|\cdot|)\times(\R,|\cdot|^{1/2})$ which can not be embedded in $\R^3$.

Using a bi-Lipschitz welding theorem \cite{Makarychev}, we obtain the following result as a corollary.

\begin{thm}\label{thm:complex}
If $M$ is a compact piecewise linear $2$-manifold in $\R^3$, then there exists $N\in\N$ such that $\Lip((M,d_K),\R^N)<\infty$.
\end{thm}

It is natural now to ask if Theorem \ref{thm:complex} holds if ``piecewise linear'' is replaced by some degree of differentiability. Our next theorem asserts that \textit{any} $C^{1,1}$ $1$-manifold can be locally bi-Lipschitz embedded in $\R^4$ with a uniform constant. Moreover we prove that the same holds true for $C^{1,1}$ $2$-manifolds around non-characteristic points. %It turns out that the situation now is more complicated and it largely  depends on the (Heisenberg) smooth structure of the manifold around each point.
Recall that a point $p$ on a differentiable manifold in $\mathbb{R}^3$ is called \emph{characteristic} if the horizontal distribution at that point is the same as the tangent space. Otherwise, the point is called \emph{$\mathbb{H}$-regular}. By well known results of Balogh \cite{Balogh}, almost all points of a $C^{1,1}$ smooth $2$-manifold in $\mathbb{R}^3$ are $\bH$-regular. A $2$-manifold is called $\mathbb{H}$-regular if all of its points are $\mathbb{H}$-regular.

\begin{thm}\label{thm:curve-surface}
There exists universal $L>1$ with the following properties.
\begin{enumerate}
\item If $\g\subset\R^3$ is a $C^{1,1}$ $1$-manifold and $x\in\gamma$, then there exists a neighborhood $U \subset \gamma$ of $x$ such that $\Lip((U,d_K),\R^4)\leq L$.
\item If $M\subset \R^3$ is a $C^{1,1}$ $2$-manifold and $x\in M$ is $\bH$-regular, then there exists a neighborhood $U \subset M$ of $x$ such that $\Lip((U,d_K),\R^4)\leq L$.
\end{enumerate}
\end{thm}

The dimension for the second part is sharp. Moreover, the embeddings are only local as there exists a $C^{1,1}$ curve $\gamma$ such that $(\g,d_{\bH})$ does not bi-Lipschitz embed into any Euclidean space (simply take a $C^{1,1}$ curve that contains $\mathbb{Z}^3$).

The case of 2-manifolds with characteristic points is more complicated, and we only give partial results. In \textsection\ref{sec:other}, we show that every $C^{1,1}$ surface obtained by revolving a curve around the $z$-axis admits local bi-Lipschitz embeddings into $\R^4$. In particular, we show that all Kor\'{a}nyi spheres bi-Lipschitz embed in $\R^4$ with a uniform constant. Such surfaces contain at most 2 characteristic points. In \textsection\ref{sec:1/2xy} we show that the surface $z=\frac12xy$, which contains infinitely many characteristic points, bi-Lipschitz embeds in $\R^{19}$.

Recall that a subset $E$ of a metric space $X$ is \emph{porous} if there exists $c\geq 1$ such that, for every $x\in E$ and every $r>0$, the set $B(x,r)\setminus E$ contains a ball of radius $r/c$. Following the techniques of Pansu and Semmes and using an Arzel\'a-Ascoli argument, one can show that a non-porous subset of $\bH$ admits no bi-Lipschitz embedding in any Euclidean or Hilbert space. As all the curves and surfaces considered above are porous sets in $\bH$, it is natural to ask if all compact porous subsets of $\bH$ admit a bi-Lipschitz embedding into some Euclidean space. In \textsection\ref{sec:porous} we answer this question negatively.

\begin{thm} \label{th:bad-porous}
  There exists a compact porous subset $X \subset \H$ that does not admit a bi-Lipschitz embedding into $\ell_2$.
\end{thm}

The paper is organized as follows.
In Section~\ref{sec:prelim} we will introduce the Heisenberg group and we will prove some basic estimates for the Kor\'{a}nyi metric which will be used repeatedly in the following.
In Section \ref{sec:affine} we will prove Theorem \ref{thm:line-plane}, while in Section \ref{sec:smoothreg} we consider smooth curves and smooth regular surfaces and provide the proof of Theorem \ref{thm:curve-surface}. In Section \ref{sec:other} we study several examples of smooth surfaces with characteristic points, and, in particular, we prove that Kor\'{a}nyi spheres bi-Lipschitz embed into $\R^4$ with a uniform constant. Finally in Section \ref{sec:porous} we use Laakso graphs in order to construct a porous compact subset of $\bH$ which does not admit a bi-Lipschitz embedding into any Euclidean space.

\section{Preliminaries}\label{sec:prelim}

Given numbers $x,y\geq 0$ and parameters $a_1,\dots, a_n$, we write $x \lesssim_{a_1,\dots,a_n} y$ if there exists a positive and finite constant $C$ depending on at most  $a_1,\dots,a_n$ such that $x \leq C y$. We write $x \simeq_{a_1,\dots,a_n} y$ when $x \lesssim_{a_1,\dots,a_n} y$ and $y \lesssim_{a_1,\dots,a_n} x$. Similarly, we write $x\lesssim y$ or $x\simeq y$ to denote that the implicit constants are universal.

Recall (see \cite{BH,Wu}) that there exists $L_{\Phi}>1$ and an embedding $\Phi: \R \to \R^3$ such that
\begin{enumerate}
\item for all $x,y\in\R$, 
\[ L_{\Phi} ^{-1}|x-y|^{1/2} \leq |\Phi(x) - \Phi(y)| \leq L_{\Phi} |x-y|^{1/2},\]
\item $\Phi([0,1]) \subset [0,1]^3$ with $\Phi(0) = (1/2,1/2,0)$ and $\Phi(1) = (1/2,1/2,1)$,
\item for all $t\in [0,1]$,
\[ \dist(\Phi(t),\partial[0,1]^3) \leq L_{\Phi}^{-1}\dist(\Phi(t),\{(1/2,1/2,0),(1/2,1/2,1)\}).\]
\end{enumerate}
That is, $\Phi$ is a bi-Lipschitz embedding of $(\mathbb{R},|\cdot|^{1/2})$ in $\mathbb{R}^3$.
%For the construction of $\Phi$ see the Appendix.

Similarly, there exists $\phi : \S^1 \to \R^3$ such that, for all $x,y\in \S^1$,
\begin{equation}
\label{snowflake2}
L_{\Phi} ^{-1}|x-y|^{1/2} \leq |\phi(x) - \phi(y)| \leq L_{\Phi} |x-y|^{1/2}.
\end{equation}

\subsection{Heisenberg group}

The \emph{Heisenberg group} is the sub-Riemannian manifold $(\R^3,H,g)$ where $H$ is the \emph{horizontal distribution} generated by the vector fields
\[ X = \partial_x - \frac12 y \partial_z \qquad\text{and}\qquad Y = \partial_y + \frac12 x\partial_z\]
and with $g$ given by
\[ g(aX + bY,cX+dY) = ac + bd.\]
The Heisenberg group can also be thought as a group $\H = (\R^3, \cdot)$ with the group law
\[ (x,y,z)\cdot(x',y',z') = \left(x+x',y+y', z+z' + \tfrac12 (xy'-x'y)\right).\]

An absolutely continuous curve $\g(t) = (x(t),y(t),z(t))$ with $t\in[0,T]$ is \emph{horizontal} 
if $\g'(t) \in H_{\g(t)}$ for almost all $t\in [0,T]$. That is,
\[ 
z'(t) + \tfrac12x'(t)y(t) - \tfrac12x(t)y'(t) = 0 \qquad\text{for almost all }t\in[0,T].
\]

The \emph{Carnot-Carath\'eodory} metric on $\mathbb{H}$ is defined by
\[ d_{cc}(p,q) = \inf_{\g}  \int_0^1 \sqrt{g(\dot{\g}(t),\dot{\g}(t))} \, dt\]
where the infimum is taken over all horizontal curves $\g:[0,1] \to \R^3$ such that $\g(0)=p$ and $\g(1)=q$.

The following example describes the horizontal curves on planes. 

\begin{ex}
In vertical planes $y = bx + c$, the horizontal curves are exactly the lines
\[ \left\{ \left(t, bt + c, -\tfrac12 ct + C_1 \right): t\in\R \right\},\qquad C_1\in\R\]
while the vertical lines $\{(0,0,t + C_2): t\in\R\}$, $C_2\in\R$ are bi-Lipschitz homeomorphic to the snowflake $(\R,|\cdot|^{1/2})$. 

On the other hand, in the plane $z=ax+by+c$, the horizontal curves are exactly the lines
which project to
\[ 
\{(-2b,2a) + t s: t\in\R\},\qquad s\in\S^1,
\]
in the $xy$-plane,
while the 
curves which project to the
concentric circles 
$$
\{(-2b,2a) + ts: s \in\S^1\},\qquad t > 0
$$ 
are bi-Lipschitz homeomorphic to re-scaled copies of the snowflake $(\S^1,|\cdot|^{1/2})$.
\end{ex}

Given points $p=(x,y,z)$ and $q=(x',y',z')$ in $\bH$, define the \emph{Kor\'anyi metric}
$$
d_K(p,q) = \left( \left((x-x')^2 + (y-y')^2\right)^2 + 16(z-z' + \tfrac12 (xy' - x'y))^2 \right)^{1/4}
$$
It is well known that $d_K$ is a metric, and it is also bi-Lipschitz equivalent to $d_{cc}$
in the sense that 
$$
C^{-1} d_K(p,q) \leq d_{cc}(p,q) \leq C d_K(p,q)
\quad
\text{for all }
p,q \in \mathbb{H}
$$ 
for some constant $C \geq 1$; see \cite{CDPT}.
More often, we consider the following distance $d_{\bH}$ (which is not a true metric since there is a sub-additive constant present in the triangle inequality):
\[ d_{\bH}(p,q) = |x-x'| + |y-y'| + \left|z-z' + \tfrac{1}{2}(xy' - x'y) \right|^{\frac12}.\]
Though $d_{\bH}$ is not a metric, it is bi-Lipschitz equivalent to $d_K$ in the above sense. 
Thus, it suffices to construct bi-Lipschitz embeddings of sets $A \subset \mathbb{R}^3$ endowed with $d_{\bH}$.
This provides a bi-Lipschitz embedding of the set $A$ when considered as a subset of the metric space $(\bH, d_K)$.
The distance $d_{\mathbb{H}}$ is related to the Euclidean metric in the following way.

\begin{lem}\label{lem:comparison}
Suppose that $R>0$ and $w,w' \in \R^3$ with $\max\{|w|,|w'|\} \leq R$. Then
\[ \min\{R^{1/2},R^{-1}\}|w-w'| \lesssim d_{\bH}(w,w') \lesssim \max\{1,R^{1/2}\}|w-w'|^{1/2}.\]
\end{lem}

\begin{proof}
Note that,
if $w=(x,y,z)$ and $w' = (x',y',z')$, then
\begin{equation}\label{eq:comparison}
d_{\bH}(w,w') = |x-x'| + |y-y'| + |z-z' + \tfrac12y'(x-x') -\tfrac12x'(y-y')|^{1/2}.
\end{equation}
Therefore, by (\ref{eq:comparison}) and the triangle inequality, we have
\begin{align*} 
d_{\bH}(w,w') &\leq |z-z'|^{1/2} + |x-x'|^{1/2}(|y'|^{1/2} + |x-x'|^{1/2}) + |y-y'|^{1/2}(|x'|^{1/2} + |y-y'|^{1/2})\\
&\lesssim (1+R^{1/2})|w-w'|^{1/2}.
\end{align*}

If $\max\{|x-x'|, |y-y'|\} \leq (2R)^{-1}|z-z'|$, then by (\ref{eq:comparison})
\begin{align*} 
d_{\bH}(w,w') \geq (|z-z'| - \tfrac12R|x-x'| -  \tfrac12R|y-y'|)^{1/2} &\gtrsim |z-z'|^{1/2} \\
&\gtrsim \min\{1,R^{1/2}\}|w-w'|^{1/2}  \\
&\geq \min\{R^{-1/2},R^{1/2}\}|w-w'|.
\end{align*}

If $\max\{|x-x'|, |y-y'|\} > (2R)^{-1}|z-z'|$, then by (\ref{eq:comparison})
\begin{align*} 
d_{\bH}(w,w') \geq |x-x'| + |y-y'| \gtrsim \min\{1,R^{-1}\}|w-w'|.
\end{align*}
The proof completes by noting that $\min\{1,R^{-1},R^{1/2},R^{-1/2}\} = \min\{R^{-1},R^{1/2}\}$.
\end{proof}

%\item Given, $\e\in(0,1)$, a curve $\g \subset M$ is a ($\e$-)\emph{cohorizontal curve} if
%\[ \e |x-y|^{1/2} \leq d_{\mathbb{H}}(x,y) \leq \e^{-1}|x-y|^{1/2}.\]
%\end{enumerate}

The length of a horizontal curve between two points can be estimated by the following lemma.

\begin{lem}\label{lem:horiz}
Let $w:[t_1,t_2] \to \bH$ be a horizontal curve with $w(t) = (x(t),y(t),z(t))$. Then,
\[ d_{\bH}(w(t_1),w(t_2)) \lesssim \left(\|x'(t)\|_{\infty} + \|y'(t)\|_{\infty}
% + \|x'(t)\|_{\infty}^{1/2}\|y'(t)\|_{\infty}^{1/2} 
\right ) |t_1-t_2|.\]
\end{lem}

\begin{proof}
We have from the absolute continuity of $x$ and $y$ that
\begin{align*} 
|x(t_1) - x(t_2)| + |y(t_1) - y(t_2)| 
= \left| \int_{t_1}^{t_2} x'(s) \, ds \right| + \left| \int_{t_1}^{t_2} y'(s) \, ds \right| \leq \left(\|x'(t)\|_{\infty} + \|y'(t)\|_{\infty}\right)|t_1-t_2|.
\end{align*}
Moreover,
\begin{align*} 
\Big|z(t_1) &- z(t_2) +\frac12x(t_1)y(t_2) - \frac12x(t_2)y(t_1) \Big|^{1/2} \\
&= 2^{-1/2} \left | \int_{t_1}^{t_2} [x'(t)y(t) - x(t)y'(t)] \, dt +x(t_1)y(t_2) - x(t_2)y(t_1)\right |^{1/2} \\
&= 2^{-1/2} \left | \int_{t_1}^{t_2} [x'(t)y(t) - x(t)y'(t)] \, dt -y(t_2)(x(t_2) - x(t_1)) + x(t_2)(y(t_2)-y(t_1))\right |^{1/2} \\
&= 2^{-1/2} \left | \int_{t_1}^{t_2} [x'(t)(y(t)-y(t_2)) \, dt - \int_{t_1}^{t_2}y'(t)(x(t)-x(t_2)) \, dt \right |^{1/2} \\
&\leq \left ( \int_{t_1}^{t_2} |x'(t)||y(t)-y(t_2)| \, dt + \int_{t_1}^{t_2}|y'(t)||x(t)-x(t_2)| \, dt \right )^{1/2}\\
&\leq  \|x'(t)\|_{\infty}^{1/2}\|y'(t)\|_{\infty}^{1/2} |t_1-t_2|.
\end{align*}
The proof completes by noting that $\|x'(t)\|_{\infty}^{1/2}\|y'(t)\|_{\infty}^{1/2} \lesssim \|x'(t)\|_{\infty} + \|y'(t)\|_{\infty}$.
\end{proof}

Given a $C^{1,1}$ $2$-manifold $M \subset\R^3$ and a point $p\in M$, we denote by $T_pM$ (resp. $H_p$) the tangent plane of $M$ (resp. the horizontal distribution) at $p$.

\begin{defn}
Let $M$ be a $C^{1,1}$ $2$-manifold. A point $p\in M$ is \emph{characteristic} if $T_pM = H_p$. A point $p\in M$ is \emph{regular} if it is not characteristic.
\end{defn}

%\emph{regular} if there exists an open set $U \subset \R^3$ containing $p$, and a $C^2$ function $F:U \to \R$ such that 
%\begin{enumerate}
%\item $M\cap U = \{w\in U : F(w) = 0\}$,
%\item $\nabla_{\mathbb{H}}F(p) \neq (0,0)$, that is,
%\[ (2F_x(p)-y F_z(p), 2F_y(p)+x F_z(p)) \neq (0,0).\]
%\end{enumerate}
%A point $p\in M$ is \emph{characteristic} if it is not regular.
%\end{defn}

In the example of vertical planes, all points are regular. In the example of the plane $z=0$, all points are regular except for the origin which is characteristic.

\begin{thm}\label{thm:regae}{\cite[Theorem 1.2]{Balogh}} Let $\mathscr{C}(M)$ be the set of characteristic points of a $2$-manifold $M$ in $\R^3$.
\begin{enumerate}
\item If $M$ is $C^{1,1}$ then $\mathscr{C}$(M) has (Euclidean) Hausdorff dimension less than $2$.
\item If $M$ is $C^{2}$ then $\mathscr{C}$(M) has (Euclidean) Hausdorff dimension less or equal to $1$.
\end{enumerate}
\end{thm}

\section{Lines, planes and PL complexes}\label{sec:affine}

In \textsection\ref{sec:lines} we show the first part of Theorem \ref{thm:line-plane} and in \textsection\ref{sec:planes} the second part. In \textsection\ref{sec:complexes} we show Theorem \ref{thm:complex}.

\subsection{Lines in the Heisenberg group}\label{sec:lines}
Here we show the first part of Theorem \ref{thm:line-plane} which we restate in the following proposition.

\begin{prop}\label{prop}
There exists universal $L>1$ such that, for any straight line $\ell \subset \R^3$, the space $(\ell,d_{\mathbb{H}})$ $L$-bi-Lipschitz embeds into $\R^3$.
\end{prop}

\begin{proof}
Suppose that $\ell$ is given by the formula
\[ w(t) = (a_1t  + c_1, a_2t + c_2, a_3 t + c_3), \qquad t\in\R \]
with $|a_1|^2 + |a_2|^2 + |a_3|^2 = 1$. Without loss of generality we may assume that $|a_2| \leq |a_1|$. Moreover, since $d_{\mathbb{H}}$ is invariant under vertical translations, we may assume that $c_3=0$. Consider the following cases.

\emph{Case 1.} Suppose that $a_1=a_2 =0$. Define the map $F:(\ell,d_{\mathbb{H}}) \to \R^3$ given by $F(w(t)) = \Phi(t)$. 
Since $|a_3|=1$,
we have for any $\tau,\tau' \in \mathbb{R}$,
\[ 
d_{\bH}(w(\tau),w(\tau')) = |\tau-\tau'|^{1/2} 
\quad
\text{and}
\quad 
|F(w(\tau))-F(w(\tau'))| = |\Phi(\tau) -\Phi(\tau')|.
\] 
Thus $F$ is $L_{\Phi}$-bi-Lipschitz where $L_{\Phi}$ is the bi-Lipschitz constant of $\Phi$.

\emph{Case 2.} Suppose that $2a_3+c_2a_1-c_1a_2 = 0$. That is, the line is a horizontal curve. Define $F:(\ell,d_{\mathbb{H}}) \to \R$ by $F(w(t)) = (|a_1|+|a_2|) t$. 
Let $\tau,\tau'\in\R$. Then
\begin{align*}
d_{\bH}(w(\tau),w(\tau')) = (|a_1|+|a_2|)|\tau-\tau'|= |F(w(\tau))-F(w(\tau'))|,
\end{align*}
so $F$ is in fact an isometry.

\emph{Case 3.} Suppose that none of the above cases applies. Set 
\[ 
\kappa = |a_1| + |a_2|
\neq 0, 
\qquad 
\lambda = \left|a_3+\tfrac12 (c_2a_1-c_1a_2) \right|^{1/2} \neq 0,
\qquad
\text{and }
\mu = \frac{\kappa^2}{\lambda^2},
\]
and define $F:(\ell,d_{\mathbb{H}}) \to \R^3$ by
\[ 
F(w(t)) = \frac{\kappa}{\mu}\Phi\left( \mu t - \left[ \mu t\right]  \right) + \frac{\kappa}{\mu} 
\left(0,0,\left[ \mu t\right] \right)
\]
where $[a]$ denotes the integer part of $a$.

Let $\tau,\tau'\in\R$, $w = w(\tau)$ and $w'=w(\tau')$. Without loss of generality, assume that $\tau'>\tau$. We have that
\[ d_{\bH}(w,w') = \kappa|\tau-\tau'| + \lambda|\tau-\tau'|^{1/2}.\]
We now divide Case 3 into three sub-cases.

\emph{Case 3.1.} 
Suppose that there exist at least two integers contained in the interval $(\mu \tau,\mu \tau')$. Let 
\[ 
m = \min \mathbb{Z}\cap (\mu\tau, \mu\tau') 
\quad
\text{and}
\quad 
n = \max \mathbb{Z}\cap (\mu\tau, \mu\tau').
\]
Then, $|\tau-\tau'| \geq \mu^{-1}$ which yields $\lambda \leq \kappa|\tau - \tau'|^{1/2}$, and thus
$$
\kappa|\tau-\tau'| 
\leq d_{\bH}(w,w') 
%= \kappa|\tau-\tau'| + \kappa \frac{\lambda}{\kappa} |\tau-\tau'|^{1/2}
\leq 2\kappa |\tau-\tau'|.
$$ 
On the other hand, since $|\mu \tau - \mu \tau'| >  1$,

\begin{align*} 
\frac{\mu}{\kappa} |F(w) - F(w')| 
&\leq \left | \Phi\left(  \mu\tau' - n \right) -  \Phi\left(  \mu\tau - (m-1)\right) \right| + (n-m+1)\\
&\lesssim  \left| \mu \tau - \mu \tau' \right|^{1/2}  +   (n-m)^{1/2} + 1   + (n-m)+1\\
&\lesssim   |\mu \tau - \mu \tau'|
= \mu|\tau-\tau'|.
\end{align*}
Notice that the third component of $\Phi$ lies between 0 and 1.
Thus we have
\begin{align*}
\frac{\mu}{\kappa} |F(w) - F(w')| 
&\geq (n-m+1) - \left | \Phi\left(  \mu\tau' - n \right) -  \Phi\left(  \mu\tau - (m-1)\right) \right| \\
&\geq (n-m) 
\gtrsim |\mu \tau - \mu \tau'| = \mu |\tau-\tau'|
\end{align*}
since $3(n-m) \geq (n-m) + 2 > |\mu \tau - \mu \tau'|$.
Hence, in this case, $|F(w)-F(w')| \simeq \kappa |\tau - \tau'| \simeq d_{\mathbb{H}}(w,w')$.

\emph{Case 3.2.} Suppose that there exists no integer in $(\mu\tau, \mu \tau')$. Then, $|\tau-\tau'| < \mu^{-1}$ which yields 
\[ \lambda|\tau-\tau'|^{1/2} \leq d_{\bH}(w,w') \leq 2\lambda|\tau-\tau'|^{1/2}.\] 
Moreover, $[\mu\tau] = [\mu\tau']$, so 
\[ 
 |F(w) - F(w')| = \frac{\kappa}{\mu} \left | \Phi\left( \mu\tau - \left[ \mu\tau\right]  \right) - \Phi\left( \mu\tau' - \left[ \mu\tau'\right]  \right) \right | 
\simeq \frac{\kappa}{\mu}|\mu\tau-\mu\tau'|^{1/2}
=\lambda|\tau-\tau'|^{1/2}. \]

\emph{Case 3.3.} Suppose that there exists exactly one integer $n$ in $(\mu\tau, \mu\tau')$. Then, $|\tau'-\tau| <2\lambda^2/\kappa^2$ which yields
\[ \lambda|\tau-\tau'|^{1/2} \leq d_{\bH}(w,w') < (1+\sqrt{2})\lambda|\tau-\tau'|^{1/2}.\]

Set $\tau'' = n / \mu \in (\tau,\tau')$ and $w'' = w(\tau'')$.
Before estimating $|F(w)-F(w')|$, we make two observations. 
Firstly, by (3) of the definition of $\Phi$,
the assumptions of this case,
and the choice of $\tau''$,
it is easy to see that 
$$
|F(w) - F(w')| \simeq |F(w)-F(w'')| + |F(w') - F(w'')|.
$$
Also, $t \mapsto \Phi(\mu t - [\mu t])$ is periodic with period $1/\mu$; 
in particular,
\[ 
|F(w(t+ 1/\mu)) - F(w(s+ 1/\mu))| = |F(w(t)) - F(w(s))|.
\]  
Combining these two observations, we have
\begin{align*} 
|F(w) - F(w')| &\simeq |F(w)-F(w'')| + |F(w') - F(w'')|\\ 
&= \left|F\left(w\left(\tau - \tfrac{n-1}{\mu} \right)\right) - F\left(w\left(\tfrac{1}{\mu} \right)\right) \right| 
	+ \left|F\left(w\left(\tau' - \tfrac{n}{\mu} \right)\right) - F\left(w\left( 0 \right)\right) \right|\\
&= \frac{\kappa}{\mu} \left |\Phi\left(   \mu \tau - (n-1) \right) - \Phi\left(  1 \right) \right | 
	+ \frac{\kappa}{\mu} \left |\Phi\left(   \mu \tau' - n \right) - \Phi\left(  0 \right) \right |\\
&\simeq \frac{\kappa}{\mu} \left( \left|n- \mu \tau\right|^{1/2} + \left| \mu \tau' - n \right|^{1/2} \right )\\
&\simeq \lambda|\tau-\tau'|^{1/2}. \qedhere
\end{align*}
\end{proof}

\subsection{Planes in the Heisenberg group}\label{sec:planes}

We now show the second part of Theorem \ref{thm:line-plane} which we restate in the following proposition.

\begin{prop}\label{thm:planes}
There exists a universal $L>1$ such that, for any plane $P$ in $\R^3$, there exists an $L$-bi-Lipschitz embedding of $(P,d_{\bH})$ into $\R^4$.
\end{prop}

For the proof of Proposition \ref{thm:planes} we distinguish two cases.

\subsubsection{Planes parallel to $z$-axis} Suppose that 
\[ P = \{ (x,y,z)\in \R^3 : y = bx + c\}\] 
for some $b,c \in \R$. Without loss of generality we may assume that $|b|<1$. Otherwise we consider the plane $P' = \{ (x,y,z)\in \R^3 : x = b^{-1}y + b^{-1}c\}$ and the proof is similar. Define the map $f: P \to \R^4$ with
\[ f(x,bx+c,z) = \left(x,\Phi \left(z-\tfrac{c}{2}x \right) \right).\]
Then, if $w=(x,bx+c,z)$ and $w'=(x',bx'+c,z')$ we have
\begin{align*}
|f(w) - f(w')| &\lesssim |x-x'| + \left|\Phi \left(z -\tfrac{c}{2}x \right) - \Phi \left(z' -\tfrac{c}{2}x' \right) \right|\\
                   &\leq |x-x'| + b|x-x'| + L_{\Phi} \left|(z-z') -\tfrac{c}{2}(x-x') \right|^{\frac12} \\
                   &= L_{\Phi} d_{\bH}(w,w').
\end{align*}
Similarly we obtain
\[|f(w)-f(w')| \geq (L_{\Phi}|b| + L_{\Phi})^{-1} d_{\bH}(w,w') \geq (2L_{\Phi})^{-1} d_{\bH}(w,w').\]

\subsubsection{Planes not parallel to $z$-axis}\label{sec:z=0} Suppose that 
\[ P = \{ (x,y,z)\in \R^3 : z = ax + by + c\}\] 
for some $a,b,c \in \R$. 

Define first $P_0 = \{(x,y,z)\in \R^3 : z=0\}$ and $g: P_0 \to P$ with 
\[ g(x,y,0) = (x-2b,y+2a,a(x-2b)+b(y+2a) + c).\]
Note that $g$ is an isometry with respect to $d_{\bH}$.
Indeed, if $w=(x,y,0)$ and $w'=(x',y',0)$ then
\[ d_{\bH}(g(w),g(w')) = |x-x'| + |y-y'| + \left|\tfrac{1}{2}(x'y-xy') \right|^{1/2} = d_{\bH}(w,w').\]
Therefore, it suffices to prove Proposition~\ref{thm:planes} for the plane $P_0$.

Given a point $w\in\R^2$ there exist $t \geq 0$ and $s \in \S^1$ such that $w = ts$; if $w\neq (0,0)$, then $t>0$ and $s$ are unique. 
Define now the embedding $F:P_0 \to \R^4$ by $F(ts,0) = (t \, \phi(s),t)$
where $\phi$ is defined as in \eqref{snowflake2}. 
The following two simple lemmas imply that $F$ is bi-Lipschitz.

\begin{lem}\label{lem:z=0}
For each $i\in\{1,2\}$ let $s_i \in \S^1$, $t_i\geq 0$ and $w_i=(s_it_i,0) \in \R^3$. Then,
\[ d_{\mathbb{H}}(w_1,w_2) \simeq |t_1-t_2| + (t_1\wedge t_2)|s_1-s_2|^{1/2}.\]
\end{lem}

\begin{proof}
Without loss of generality assume that $t_1\leq t_2$. For $i=1,2$ write $s_i=(\cos(\theta_i),\sin(\theta_i))$ for some $\theta_i \in [0,2\pi)$. Then,
\begin{align*}
d_{\mathbb{H}}(w_1,w_2) &\simeq |s_1t_1 - s_2t_2| + (t_1t_2)^{1/2}|\sin(\theta_1 - \theta_2)|^{1/2}\\
&\simeq (t_2-t_1) + t_1|s_1-s_2| + (t_1t_2)^{1/2}|\sin(\theta_1 - \theta_2)|^{1/2}.
\end{align*}

If $t_2\geq 2t_1$, then $ t_1 \leq (t_2-t_1)$, so $(t_1t_2)^{1/2} \lesssim (t_2-t_1)$ and the claim holds.

If $t_2< 2t_1$ and $(\theta_1 - \theta_2) > \pi/2$, then $|s_1-s_2|\simeq 1$, $t_1\simeq t_2$ and the claim holds.

If $t_2< 2t_1$ and $(\theta_1 - \theta_2) \leq \pi/2$, then $|s_1-s_2|\simeq \sin(\theta_1 - \theta_2)$, $t_1\simeq t_2$, and $|s_1-s_2| \lesssim |s_1-s_2|^{1/2}$, and so the claim holds.
\end{proof}

\begin{lem}\label{lem:snowcone}
If $t_1,t_2 \geq 0$ and $s_1,s_2\in \S^1$ then
\begin{align*}
|(t_1\phi(s_1),t_1) - (t_2\phi(s_2),t_2)|  \simeq |t_1-t_2| + (t_1\wedge t_2)|s_1-s_2|^{1/2} 
\end{align*}
\end{lem}

\begin{proof}
Indeed,
\begin{align*}
|(t_1\phi(s_1),t_1) - (t_2\phi(s_2),t_2)|  
&\simeq |t_1-t_2| + |t_1\phi(s_1) - t_2\phi(s_2)|\\
&\simeq |t_1-t_2| + (t_1\wedge t_2)|\phi(s_1) - \phi(s_2)|\\
&\simeq |t_1-t_2| + (t_1\wedge t_2)|s_1-s_2|^{1/2}. \qedhere
\end{align*}
\end{proof}
%Moreover, $ |(t\phi(p),t) - (t'\phi(p'),t')| \simeq t\vee t' $ if $ t\vee t' \geq \frac12 t\wedge t'$

\subsection{Simplicial complexes}\label{sec:complexes}

Recall that a $0$-simplex is a point of $\R^3$, a $1$-simplex $E\subset \R^3$ is the convex hull of two distinct points and a $2$-simplex $E\subset \R^3$ is the convex hull of three distinct points in $\R^3$ that do not lie on the same line.  The boundary of a $k$-simplex consists of $j$-faces ($j\leq k-1$) which are $j$-simplices.

A (compact) simplicial $2$-complex $K\subset \R^3$ is a finite union of simplices $S_1\cup\dots\cup S_k$ with the following properties
\begin{enumerate}
\item each $S_i$ is a $k_i$-simplex for some $k_1\in\{0,1,2\}$;
%\item every face of a simplex in $\mathcal{K}$ is contained in $\mathcal{K}$; 
\item for each $i,j \in \{1,\dots,n\}$, either $S_i \cap S_j$ is empty, or $S_i \cap S_j$ is a common edge of $S_i$ and $S_j$, or $S_i=S_j$.
\end{enumerate} 

\begin{thm}\label{thm2}
%Let $\mathcal{C}\subset \R^3$ be a finite simplicial complex. Then, there exists a bi-Lipschitz embedding $f: (\mathcal{C},d_K) \to \R^4$.
Let $K\subset \R^3$ be a finite union $K = D_1\cup\cdots\cup D_k$ where each $D_i$ is a simplicial $2$-complex. Then, there exists $N\in\N$ and a bi-Lipschitz embedding of 
$(K,d_{\bH})$ into $\R^N$.
\end{thm}

For the proof of Theorem \ref{thm2} we use the following bi-Lipschitz welding theorem.

\begin{lem}[{\cite[Theorem 1.1]{Makarychev}}]\label{lem:welding}
Let $X$ be a metric space and let $X_1,X_2 \subset X$ be closed subsets such that $X=X_1\cup X_2$. If $X_1$ bi-Lipschitz embeds in $\R^{n}$ and $X_2$ bi-Lipschitz embeds in $\R^{m}$ then $X$ bi-Lipschitz embeds in $\R^{n+m+1}$.
\end{lem}

We now show Theorem \ref{thm2}.

\begin{proof}[{Proof of Theorem \ref{thm2}}]
By Lemma \ref{lem:welding} we only need to show that a simplicial $2$-complex admits a bi-Lipschitz embedding into some $\R^N$. Let $K\subset \R^3$ be a simplicial $2$-complex. Then, $K = \bigcup_{i=1}^n S_i$ where each $S_i$ is a $k_i$-simplex for some $k_i\in\{0,1,2\}$. By Proposition \ref{thm:planes}, each $(S_i,d_{\bH})$ $L$-bi-Lipschitz embeds into $\R^4$ for some universal $L>1$. Applying Lemma \ref{lem:welding} $n$ times we obtain a bi-Lipschitz embedding of $(K,d_{\bH})$ into $\R^{5n-1}$.
\end{proof}

\section{Smooth curves and $\bH$-regular surfaces}\label{sec:smoothreg}

In \textsection\ref{sec:curves} we show the first part of Theorem \ref{thm:curve-surface}, and in \textsection\ref{sec:surfaces} we prove the second part.
For a set $A \subset \mathbb{R}^3$,
we will say that $(A,d_{\mathbb{H}})$ {\em locally} $L$-bi-Lipschitz embeds in $\mathbb{R}^N$ for some $L>1$ if,
for any $p \in A$, there is a neighborhood $U$ of $p$ such that
$(U \cap A,d_{\mathbb{H}})$ $L$-bi-Lipschitz embeds in $\mathbb{R}^N$.

\subsection{Smooth curves}\label{sec:curves}

For the rest of \textsection\ref{sec:curves}, we denote by $I$ a non-degenerate interval of $\R$.
% and a $C^1$ curve
%\[ w(t) = (x(t),y(t),z(t)) \qquad\text{with}\quad t\in I\]
%with unit speed, i.e., $(x'(t))^2+(y'(t))^2+(z'(t))^2 = 1$ for all $t\in I$. 
We prove the following proposition.

\begin{prop}\label{thm:curves}
There is a universal constant $L > 1$ such that,
for any $C^{1,1}$ curve $w:I \to \R^3$ with unit speed,
$(w(I),d_{\mathbb{H}})$ locally $L$-bi-Lipschitz embeds into $\R^4$.
%There exists $N\in\N$ and a bi-Lipschitz embedding $f:(\g,d_K) \to \R^N$.
\end{prop}

The idea of the proof is to decompose the curve $\g = w(I)$ into two types of sub-arcs and embed each one in $\R^4$. In the following lemma we consider sub-arcs of $\gamma$ that are ``vertical" enough.

\begin{lem}\label{lem:snow}
Let $w:[0,T] \to \R^3$ be a $C^{1,1}$ smooth curve with unit speed such that $|2z'(t)+x'(t)y(t)-y'(t)x(t)|\geq 1$ for all $t\in[0,T]$. Then, $(w([0,T]),d_{\bH})$ locally $L$-bi-Lipschitz embeds into $\R^3$ for some universal $L>1$.
\end{lem}

\begin{proof}
Let $M>0$ be such that, for all $t\in [0,T]$,
\[1 \leq |2z'(t) + x'(t)y(t) - x(t)y'(t)| \leq M.\]
%and $\max\{|x(t)|, |y(t)|, |z(t)|\} \leq M$.

By compactness of $[0,T]$ and uniform continuity of $w'(t)$, there exists $\e \in (0,1)$ such that if $t_0,\dots,t_5\in [0,T]$ are such that $|t_i-t_j|\leq \e$, then
\begin{equation}\label{eq:1}
\frac12 \leq \frac{|2z'(t_1) + x'(t_2)y(t_3) - x(t_4)y'(t_5)|}{|2z'(t_0) + x'(t_0)y(t_0) - x(t_0)y'(t_0)|} \leq 2.
\end{equation}

We claim that,
for any $t,t' \in [0,T]$ with $|t-t'| < \e$, 
there is an $L$-bi-Lipschitz map between $(w([t,t']),d_{\mathbb{H}})$ and $([t,t'],\lambda |\cdot|^{1/2})$
for some constant $\lambda \geq 1$
and a universal constant $L > 1$.
Assuming the claim, we see that $(w([0,T]),d_{\bH})$ locally $L$-bi-Lipschitz embeds into $\R^3$ for some universal $L$. 

To prove the claim, fix $\tau_1,\tau_2 \in [0,T]$ such that $0<\tau_2-\tau_1\leq \e$. Fix also $t_0 \in [\tau_1,\tau_2]$ and set
\[ \lambda = |2z'(t_0) + x'(t_0)y(t_0) - x(t_0)y'(t_0)|^{1/2} \in [1,M^{1/2}].\]
By the Mean Value Theorem, \eqref{eq:1}, 
and the fact that $\e < 1 \leq \lambda$,
if $t_1,t_2 \in [\tau_1,\tau_2]$, there exist $c_1,\dots,c_5 \in (t_1,t_2)$ such that
\begin{align*}
d_{\bH}(w(t_1),&w(t_2))\\
&= |x(t_1) - x(t_2)| + |y(t_1) - y(t_2)| + |z(t_1) - z(t_2) + \tfrac{1}{2}(x(t_1)y(t_2)-x(t_2)y(t_1))|^{1/2}\\
&= (|x'(c_1)|+|y'(c_2)|)|t_1-t_2| +|z'(c_3) + \tfrac12x'(c_4)y(t_2) - \tfrac12 x(t_2)y'(c_5)||t_1-t_2|^{1/2}\\
&\simeq |z'(t_0) + \tfrac12x'(t_0)y(t_0) - \tfrac12 x(t_0)y'(t_0)||t_1-t_2|^{1/2}\\
&= \lambda |t_1-t_2|^{1/2}.
\end{align*}
Hence, $(w([\tau_1,\tau_2]),d_{\bH})$ is $L$-bi-Lipschitz equivalent to $([\tau_1,\tau_2], \lambda |\cdot|^{1/2})$ and the claim follows.
\end{proof}

\begin{proof}[{Proof of Proposition \ref{thm:curves}}]
Fix $t_1,t_2\in I$ with $t_1<t_2$. We show that the space $(w([t_1,t_2]),d_{\bH})$ locally bi-Lipschitz embeds in $\R^4$. 

Let $M>1$ be such that 
\[ \max\{|x(t)|,|y(t)|\} \leq M \qquad\text{for all }t\in[t_1,t_2].\]
Let $\e_0 = (4M)^{-2}$. Note that if $|x'(t)|^2 +|y'(t)|^2 \leq \e_0$, then, since $w$ has unit speed in $\mathbb{R}^3$, we have
\begin{align}\label{eq:2}
|2z'(t) + x'(t)y(t) - x(t)y'(t)| &\geq 2|z'(t)| -|x'(t)||y(t)| - |x(t)||y'(t)|\\
&\geq 2|z'(t)| - 2M\sqrt{\e_0}\notag\\
&= 2\sqrt{1-|x'(t)|^2-|y'(t)|^2} - 2M\sqrt{\e_0} \notag\\
&\geq 2\sqrt{1 - \e_0} - 2M\sqrt{\e_0} \notag\\
&\geq 1. \notag
\end{align}

By uniform continuity of $w'|_{[t_1,t_2]}$ and compactness of $[t_1,t_2]$, there exist finitely many closed intervals $I_1,\dots, I_k \subset [t_1,t_2]$ such that 
\begin{enumerate}
\item $\bigcup_{j=1}^k I_j = [t_1,t_2]$,
\item for each $j\in\{1,\dots,k\}$, either $|x'(t)|^2 +|y'(t)|^2 \leq \e_0$ for all $t\in I_j$ or $|x'(t)|^2 +|y'(t)|^2 \geq \e_0/2$ for all $t\in I_j$.
\end{enumerate}

If $|x'(t)|^2 +|y'(t)|^2 \leq \e_0$ for all $t\in I_j$, then (\ref{eq:2}) and Lemma \ref{lem:snow} imply that $(w(I_j),d_{\bH})$ locally bi-Lipschitz embeds in $\R^3$.

If $|x'(t)|^2 +|y'(t)|^2 \geq \e_0/2$ for all $t\in I_j$, then the projection of $w(I_j)$ on $\R^2\times\{0\}$ is a $C^{1,1}$ smooth curve which can be re-parameterized so that $|x'(t)|^2 + |y'(t)|^2 = 1$. 
The surface $(w(I_j)\times\R,d_{\bH})$ is $C^{1,1}$ and consists only of $\mathbb{H}$-regular points.
Hence, we will see in Proposition~\ref{thm:reg} that $(w(I_j)\times\R,d_{\bH})$ locally bi-Lipschitz embeds in $\R^4$
with a universal constant.
In particular, this means that $(w(I_j),d_{\bH})$ locally bi-Lipschitz embeds in $\R^4$.
\end{proof}

\subsection{$\bH$-regular surfaces}\label{sec:surfaces}
Here we show the second part of Theorem \ref{thm:curve-surface} which we restate in the following proposition.

%Just how odd are the characteristic points? See Section \ref{sec:axy}!!

%\begin{defn}
%A surface $S\subset \R^3$ is a \emph{$\mathbb{H}$-regular surface} if for any $p\in S$ there exists an open set $U$ and a function $F:U \in C^2(U,\R)$ such that
%\begin{enumerate}
%\item $S\cap U = \{q\in U : F(q)=0\}$
%\item $(2F_x - y F_z, 2F_y + xF_z) \neq (0,0)$  for all $p\in U$ .
%\end{enumerate}
%\end{defn} \note{Slightly different definition because of $C^2$}

\begin{prop}\label{thm:reg}
There exists a universal constant $L>1$ satisfying the following.
Let $M\subset \R^3$ be a $C^{1,1}$ $2$-manifold in $\mathbb{R}^3$ and let $p_0$ be an $\bH$-regular point of $M$. 
Then there exist a neighborhood $U \subset \mathbb{R}^3$ of $p_0$ and an $L$-bi-Lipschitz embedding of
$(U \cap M, d_{\mathbb{H}})$ in $\R^4$.
\end{prop}

The proof of this proposition is comprised by the following two lemmas.
In Lemma~\ref{lem:foliation}, 
we show that $M$ can be foliated near $p_0$ by horizontal curves
with $C^1$ dependence on parameters.
In Lemma~\ref{lem:surfaces},
we construct a local embedding of any surface which possesses such a foliation.
In fact, Lemma~\ref{lem:surfaces} will be used in several other arguments throughout the rest of this paper
when a foliation by horizontal curves can be constructed.

\begin{lem}\label{lem:foliation}
There exists a universal constant $L'>1$ satisfying the following.
Let $M$ be a $C^{1,1}$ $2$-manifold in $\R^3$ and let $p_0$ be an $\bH$-regular point of $M$. Then there exist $\e>0$ and an $L'$-bi-Lipschitz map $G:[-\e,\e]^2 \to (M,|\cdot|)$ such that 
\begin{enumerate}
\item $G(0,0) = p_0$
\item for all $v\in [-\e,\e]$, the curve $u \mapsto G(u,v)$ is a horizontal curve.
\end{enumerate}
\end{lem}

\begin{proof}
In a neighborhood of $p_0=(x_0,y_0,z_0)$, the manifold $M$ can be written as the level set $H(x,y,z) = 0$ for some $C^{1,1}$ function $H$ with 
$|\nabla H(x_0,y_0,z_0)| >0$. Without loss of generality, we assume that 
\[ |H_z(x_0,y_0,z_0)| \geq \max\{|H_x(x_0,y_0,z_0)|,|H_y(x_0,y_0,z_0)|\}.\]
The other two cases are similar and are left to the reader.

Then, in a sufficiently small neighborhood $U$ of $p_0$ we can write 
\[ M\cap U = \{(x,y,F(x,y)) : (x,y) \in V\}\]
for some open $V \subset \mathbb{R}^2$
where $F$ is a $C^{1,1}$ function with $F(x_0,y_0) = z_0$ and 
\[ \sup_{(x,y)\in V}\{|F_x(x,y)|, |F_y(x,y)|\} \leq 2.\]

Since $p_0$ is $\bH$-regular, either $x_0 - 2 F_y(x_0,y_0) \neq 0$ or $y_0+2F_x(x_0,y_0) \neq 0$. Without loss of generality, we assume that $x_0-2 F_y(x_0,y_0) \neq 0$ and that
\begin{equation}
\label{ODEbnd} 
|y_0+2F_x(x_0,y_0)| \leq |x_0-2F_y(x_0,y_0)|.
\end{equation}
There exists a neighborhood $V_1\subset V$ of $(x_0,y_0)$ such that 
$x-2F_y \neq 0$ in $V_1$.
Thus the horizontal curves in a neighborhood of $p_0$ in $M$ are given by the ODE
\begin{equation}\label{eq:ODE1}
\frac{dy}{dx}  = \frac{y+2F_x}{x-2F_y}.
\end{equation}
By standard existence and uniqueness theorems (e.g.  \cite[Theorem 1.2]{ODE} and \cite[Theorem 2.2]{ODE}), there is some $\e_1>0$ such that, for each $v \in [-\e_1,\e_1]$, there exists a unique solution $y = g_v(x)$ of (\ref{eq:ODE1}) defined on $[x_0-\e_1,x_0+\e_1]$ such that $g_v(x_0) = y_0 + v$.  Moreover, the function 
\[ f:[-\e_1,\e_1]^2 \to \R^2 \qquad\text{with}\qquad f(u,v) := (x_0+u,g_v(x_0+u))\] 
%given by 
%\[ G(x,t) := (x,g_t(x))\] 
is $C^{1}$ \cite[Theorem 7.1]{ODE}. 
By existence and uniqueness of solutions, 
there is a neighborhood of $(x_0,y_0)$ in which the graphs of $g_v$ do not intersect,
and every point in this neighborhood is contained in the graph of some $g_v$. Moreover, the derivative of $f$ at $(0,0)$ satisfies
\[ Df(0,0) = 
\begin{bmatrix}
    1       & 0 \\
    \frac{d}{du}g_{0}(x_0 + u)|_{u=0}      &  \frac{d}{dv}g_{v}(x_0)|_{v=0} 
\end{bmatrix}
=
\begin{bmatrix}
    1       & 0 \\
    \frac{y_0+2F_x(x_0,y_0)}{x_0-2F_y(x_0,y_0)}       & 1 
\end{bmatrix} .
\]
Thus, for some $\e_2 > 0$,
$G$ is a $C^1$ diffeomorphism in 
$[-\e_2,\e_2]^2$,
and, by \eqref{ODEbnd}, there exists a universal constant $L_1>1$ such that
\begin{equation}
\label{lbound} 
L_1^{-1} |(u,v)-(u',v')| \leq |f(u,v) - f(u',v')| \leq L_1 |(u,v)-(u',v')|
\end{equation}
 for all $(u,v), (u',v')\in [-\e_2,\e_2]^2$.
 
Define now $G: [-\e_2,\e_2]^2 \to M$ by $G(u,v) = (f(u,v), F(f(u,v)))$ which has been defined so that the curves $u \mapsto G(u,v)$ are horizontal curves contained in $M$ for any $v \in [-\e_2,\e_2]$. Moreover
\begin{align*}
L_1^{-1}|(u,v)-(u',v')| &\leq |G(u,v) - G(u',v')| \\
&\lesssim (L_1 + \|F_x\|_{\infty} + \|F_y\|_{\infty}) |(u,v)-(u',v')|\\
&\leq (L_1+2)|(u,v)-(u',v')|,
\end{align*}
 and hence $G$ is $L'$-bi-Lipschitz for some universal $L'>1$.
\end{proof}

\begin{lem}
\label{lem:surfaces}
There exists a universal constant $L>1$ satisfying the following.
Fix $M \subset \mathbb{R}^3$.
Suppose there exist $\e>0$, $L'>1$, and an $L'$-bi-Lipschitz map 
$G:[-\e,\e]^2 \to M$ such that 
\begin{enumerate}
\item $G(0,0) = p_0$
\item for all $v\in [-\e,\e]$, the curve $u \mapsto G(u,v)$ is a horizontal curve.
\end{enumerate}
Write $G = (g_1,g_2,g_3)$, and set
\begin{align*} 
\kappa &:= |2\partial_v g_3(0,0) + g_2(0,0)\partial_v g_1(0,0) - g_1(0,0)\partial_v g_2(0,0)|\\
\lambda &:= |\partial_u g_1(0,0)|+|\partial_u g_2(0,0)|.
\end{align*}
Define the map
\[ 
\Psi : (G([-\e,\e]^2) \cap M, d_{\bH}) \to \R^4 \quad\text{with}\quad \Psi(G(u,v)) = (\lambda u,\kappa^{1/2}\Phi(v)).
\]
Then 
there is a neighborhood $U \subset \mathbb{R}^3$ of $p_0$
such that $\Psi: (U \cap M, d_{\bH}) \to \mathbb{R}^4$ is $L$-bi-Lipschitz.
\end{lem}

Note that, while $L$ may be found independent of the bi-Lipschitz constant $L'$,
the size of the neighborhood $U$ certainly depends on the value of $L'$.

\begin{proof}
We first note that $\lambda>0$ since the curve $\gamma_1(u) = G(u,0)$ is horizontal. 
Moreover, since $G$ is $L'$-bi-Lipschitz, $\lambda$ is bounded from above by a constant $C$ depending only on $L'$. 
On the other hand, the curve $\gamma_2(v) = G(0,v)$ intersects $\gamma_1$ transversely at $p_0$ and hence cannot be horizontal. Thus, $\kappa>0$. 
Fix $(u_1,v_1),(u_2,v_2) \in [-\e,\e]^2$ and set
\[ p_1 = G(u_1,v_1), \quad p_2=G(u_2,v_2), \quad p_3=G(u_1,v_2).\]
That is, $p_2$ and $p_3$ lie along the same horizontal curve $u \mapsto G(u,v_2)$.

Choosing $\e$ possibly smaller, we have for some $c_1,c_2,c_3 \in [-\e,\e]$
\begin{align*}
d_{\bH}(p_1,p_3) &= |g_1(u_1,v_1)-g_1(u_1,v_2)| + |g_2(u_1,v_1)-g_2(u_1,v_2)|\\
&\quad + \left | g_3(u_1,v_1) - g_3(u_1,v_2) +\tfrac12 g_1(u_1,v_1)g_2(u_1,v_2) - \tfrac12 g_2(u_1,v_1)g_1(u_1,v_2) \right |^{1/2}\\
&= (|\partial_v g_1(u_1,c_1)| + |\partial_v g_1(u_2,c_2)|)|v_1-v_2|\\
&\quad + \left | \partial_v g_3(u_1,c_3) +\tfrac12 \partial_v g_1(u_1,c_1)g_2(u_1,v_2) - \tfrac12 \partial_v g_2(u_1,c_2)g_1(u_1,v_2) \right |^{1/2}|v_1-v_2|^{1/2}\\
&\lesssim \lambda|u_1-u_2| + \kappa^{1/2}|v_1-v_2|^{1/2}.
\end{align*}
Additionally, by Lemma \ref{lem:horiz} we have
\begin{align}\label{eq:regest2}
d_{\bH}(p_2,p_3) \lesssim (\|\partial_u g_1\|_{\infty}+\|\partial_u g_2\|_{\infty})|u_1-u_2| \lesssim \lambda|u_1-u_2|
\end{align}
assuming that $\e$ is sufficiently small. By the triangle inequality,
\[ d_{\bH}(p_1,p_2) \lesssim \lambda|u_1-u_2| + \kappa^{1/2}|v_1-v_2|^{1/2} \simeq |\Psi(p_1) - \Psi(p_2)|.\]

We now show that $d_{\mathbb{H}}(p_1,p_2) \gtrsim |\Psi(p_1) - \Psi(p_2)|$. By (\ref{eq:regest2}), there exists a universal constant $C_0$ such that
\begin{equation}\label{eq:regest3}
\left| g_3(u_1,v_2)-g_3(u_2,v_2) + \tfrac12 g_1(u_1,v_2)g_2(u_2,v_2) - \tfrac12 g_2(u_1,v_2)g_1(u_2,v_2) \right|^{1/2} \leq C_0\lambda|u_1 - u_2|.
\end{equation} 
Consider the following two cases

\emph{Case 1.} Suppose that $C_0 \lambda |u_1-u_2| \geq (\kappa/16)^{1/2}|v_1-v_2|^{1/2}$.
\begin{align*}
|\Psi(p_1) - \Psi(p_2)| \simeq \lambda|u_1-u_2| + \kappa^{1/2}|v_1-v_2|^{1/2} \lesssim \lambda |u_1-u_2|\lesssim d_{\bH}(p_1,p_2).
\end{align*}

\emph{Case 2.} Suppose that $C_0 \lambda |u_1-u_2| < (\kappa/16)^{1/2}|v_1-v_2|^{1/2}$. Then,
\[ d_{\bH}(p_1,p_2) \geq \lambda|u_1-u_2| + |g_3(u_1,v_1)-g_3(u_2,v_2) + \tfrac12 g_1(u_1,v_1)g_2(u_2,v_2) - \tfrac12 g_2(u_1,v_1)g_1(u_2,v_2)|^{1/2}.\]
Now,
\begin{align*}
g_3(u_1,v_1)&-g_3(u_2,v_2) + \tfrac12 g_1(u_1,v_1)g_2(u_2,v_2) - \tfrac12 g_2(u_1,v_1)g_1(u_2,v_2)\\
&= g_3(u_1,v_1)-g_3(u_1,v_2) + \tfrac12 g_2(u_2,v_2)(g_1(u_1,v_1)-g_1(u_1,v_2))\\
&\quad - \tfrac12 g_1(u_2,v_2) (g_2(u_1,v_1)-g_2(u_1,v_2))\\
&\quad +g_3(u_1,v_2)-g_3(u_2,v_2) + \tfrac12 g_1(u_1,v_2)g_2(u_2,v_2) - \tfrac12 g_2(u_1,v_2)g_1(u_2,v_2)
\end{align*}
Again assuming that $\e$ is sufficiently small,
\begin{align*}
\Big|&g_3(u_1,v_1)-g_3(u_1,v_2) + \tfrac12 g_2(u_2,v_2)(g_1(u_1,v_1)-g_1(u_1,v_2))\\
& \qquad - \tfrac12 g_1(u_2,v_2) (g_2(u_1,v_1)-g_2(u_1,v_2))\Big|\\
&= \left|\partial_v g_3(u_1,c_3) +\tfrac12 g_2(u_2,v_2)\partial_v g_1(u_1,c_1) - \tfrac12 g_1(u_2,v_2)\partial_v g_2(u_1,c_2)\right||v_1-v_2|\\
&\geq \kappa/2 |v_1-v_2|.
\end{align*}
Thus, by (\ref{eq:regest3})
\begin{align*}
d_{\bH}(p_1,p_2) &\gtrsim \lambda|u_1-u_2| + \left( \kappa/2 \right)^{1/2}|v_1-v_2|^{1/2} - C_0\lambda|u_1-u_2|\\
&\gtrsim \lambda|u_1-u_2| + \kappa^{1/2}|v_1-v_2|^{1/2}\\
&\simeq |\Psi(p_1)-\Psi(p_2)|.
\end{align*}
Therefore, $\Psi$ is bi-Lipschitz in a neighborhood of $p_0$.
\end{proof}

%\begin{rem}\label{rem:flag}
%Suppose now that, in a neighborhood of a regular point, we can express $M$ as the graph of a $C^{1,1}$ function $x=F(y,z)$. Then the point $p=(F(y_0,z_0),y_0,z_0)$ satisfies
%\[ F(y_0,z_0) - y_0F_y(y_0,z_0) \neq 0 \qquad\text{or}\qquad 2+y_0F_z(y_0,z_0)  \neq 0.\]
%Without loss of generality assume the latter. Then, the horizontal curves on $M$, close to $p$ are given by the ODE
%\begin{equation}\label{eq:ODE2}
%\frac{dz}{dy} = \frac{F-yF_y}{2+y F_z} .
%\end{equation}
%As in the proof of Proposition~\ref{thm:reg} in the case $z=F(x,y)$, there exists a map 
%\[ G:[y_0-\e,y_0+\e]\times[z_0-\e,z_0+\e] \to \mathbb{R}^2\]
%such that, for each $t\in [z_0-\e,z_0+\e]$, $g_t(y) = G(y,t)$ is a $C^1$ solution of (\ref{eq:ODE2})
%and $G$ is a $C^1$ diffeomorphism in some (possibly smaller) neighborhood $U$ of $(y_0,z_0)$. 
%Define now the embedding
%\[ \Psi:\{(F(y,z),y,z) : (y,z) \in U\} \to \R^4\]
%given by
%\[ \Psi(F(y,z),y,z) = (\Phi(G(y_0,t)),y) \]
%where $t\in[z_0-\e,z_0+\e]$ is such that $G(y,t)=z$. Working as above, we can show that $\Psi$ is bi-Lipschitz. The case $y=F(x,z)$ is identical.
%\end{rem}

\subsection{The smooth torus}
Although the set of characteristic points of a given smooth manifold is very small, it rarely is empty. By the Hairy-Ball Theorem and the classification of $2$-manifolds, if a compact differentiable manifold $M$ has no characteristic points, then it must be homeomorphic to the torus $\S^1\times\S^1$.  On the other hand, a topological torus could be $\bH$-regular. Given numbers $0<r<R$, define the torus
\[ \mathbb{T}_{r,R} = \{(x,y,z) : (R-\sqrt{x^2+y^2})^2 +z^2 = r^2\}.\]

We claim that $\mathbb{T}_{r,R}$ has no characteristic points. Recall first that both the horizontal distribution and the tangent space of $\mathbb{T}_{r,R}$ are invariant under rotations with respect to the $z$-axis in the following sense. If $P_w$ (resp. $\hat{P}_w$) is the affine space containing $w\in \mathbb{T}_{r,R}$ and generated by $T_w\mathbb{T}_{r,R}$ (resp. generated by $H_w$) and if $\zeta:\R^3 \to\R^3$ is a rotation with respect to $z$-axis, then $\zeta(P_w) = P_{\zeta(w)}$ and $\zeta(\hat{P}_w) = \hat{P}_{\zeta(w)}$. To show the claim, suppose that a point $w \in \mathbb{T}_{r,R}$ is characteristic and write $w = (r_0,\theta_0,z_0)$ in cylindrical coordinates. By the rotation invariance of the tangent space and the horizontal distribution, every point of the circle $S = \{(r_0,\theta,z_0) : \theta \in [0,2\pi)\}$ is characteristic. Therefore, $S$ is horizontal. However, by \textsection\ref{sec:planes}, $(S,d_{\bH})$ is bi-Lipschitz equivalent to $(\S^1,|\cdot|^{1/2})$, which is a contradiction.

Let $V$ be a smooth unitary vector field on $\mathbb{T}_{r,R}$ such that $V(w) \in T_w\mathbb{T}_{r,R} \cap H_w$ for all $w\in \mathbb{T}_{r,R}$. By the Frobenius Theorem, $V$ is integrable and $\mathbb{T}_{r,R}$ is foliated by the (closed) integral curves of $V$ which are obtained by rotating one of them with respect to the $z$-axis. Specifically, there exists a diffeomorphism $G:\S^1\times \S^1 \to \mathbb{T}_{r,R}$ such that
$s \mapsto G(s,\eta)$
%$\g_{\eta} : \S^1 \to  \mathbb{T}_{r,R}$ with $\g_{\eta}(s) = G(s,\eta)$ 
is a horizontal curve for each $\eta \in \S^1$.

We now claim that $(\mathbb{T}_{r,R},d_{\bH})$ bi-Lipschitz embeds in $\R^4$. 
Indeed, say $\tau$ is an identification of $\mathbb{R}^2 / \mathbb{Z}^2$ with $\S^1 \times \S^1$
and define $\Psi:\mathbb{T}_{r,R} \to \mathbb{R}^4$
so that 
$$
\Psi(G(\tau(u,v))) = (u,\Phi(v))
\quad
\text{for all }
(u,v) \in \mathbb{R}^2 / \mathbb{Z}^2.
$$
According to Lemma~\ref{lem:surfaces}, there is a constant $L>1$ such that, for any $p \in \mathbb{T}_{r,R}$, 
there exists a neighborhood $U_p$ of $p$ in $\mathbb{T}_{r,R}$ wherein $\Psi|_{(U_p,d_{\bH})}$ is $L$-bi-Lipschitz.
%Let $\phi = (\phi_1,\phi_2,\phi_3):\S^1 \to \R^3$ be the snowflaking embedding \eqref{snowflake2} with $\min_{s\in\S^1}\phi_3(s) = 1$. 
%We now rotate the snowflake $\phi(\S^1)$ in $\R^4$ with respect to the plane $\R^2\times\{(0,0)\}$ as follows. 
%Define
%\[ \Psi : (\mathbb{T}_{r,R},d_{\bH}) \to \R^4 \qquad\text{with}\qquad \Psi(G(\eta,s)) = (\phi_1(\eta),\phi_2(\eta),\phi_3(\eta)s).\]
%By Proposition \ref{thm:reg}, there exists universal $L>1$ such that, for each $p \in \mathbb{T}_{r,R}$, there exists a neighborhood $U_p$ of $p$ in $\mathbb{T}_{r,R}$ such that $\Psi$ is $L$-bi-Lipschitz. 
By the compactness of $\mathbb{T}_{r,R}$, there exists $\e>0$ such that $\Psi$ is bi-Lipschitz on all compact sets of $\mathbb{T}_{r,R}$ with (Euclidean) diameter at most $\e$. 
On the other hand, by Lemma \ref{lem:comparison}, for all $w,w' \in \mathbb{T}_{r,R}$ with $|w-w'|\geq \e$ we have
\[ d_{\bH}(w,w') \simeq_{\e} 1 \simeq_{\e} |\Psi(w) - \Psi(w')|.\]
Therefore, $\Psi$ is bi-Lipschitz.

\section{Some smooth surfaces with characteristic points}\label{sec:other}
Here we examine the bi-Lipschitz embeddability of some special classes of smooth surfaces that contain characteristic points.

\subsection{Surfaces obtained by revolution about the $z$-axis}\label{sec:revolution}
Let $F:[0,\infty) \to \mathbb{R}$ be a $C^{1,1}$ function. Let $\Sigma$ be the surface of revolution generated by $F$:
\[ \Sigma = \{(w,F(|w|)) : w\in\R^2\}.\] 
Below we denote by $\textbf{0}$ the origins $(0,0)$ and $(0,0,0)$, depending on the context.

\begin{thm}\label{thm:revolution}
Suppose that $F$ is $C^{1,1}$ with $F'(0) = 0$,
and set
%the following hold.
%\begin{enumerate}
%\item 
\[M:=\max\left\{1,  \limsup_{t\downarrow 0} \frac{|F'(t)|}{t} \right\}.\]
There exist $\e>0$, $L=L(M)>1$,
and a neighborhood $U \subset \Sigma$ of $(0,0,F(0))$
such that $(U,d_{\mathbb{H}})$ $L$-bi-Lipschitz embeds in $\R^4$.
%\item Any compact subset of $(\Sigma,d_{\mathbb{H}})$ bi-Lipschitz embeds in some Euclidean space $\R^n$.
%\end{enumerate}
\end{thm}

Before proving this theorem, we observe an application of this embedding.
Combined with Lemma \ref{lem:welding} and the $\bH$-regularity of $\Sigma\setminus\{(0,0,F(0))\}$, Theorem \ref{thm:revolution} yields the following corollary.

\begin{cor}
Suppose that $F$ is $C^{1,1}$ with $F'(0) = 0$. Any compact subset of $(\Sigma,d_{\mathbb{H}})$ bi-Lipschitz embeds in some Euclidean space $\R^n$.
\end{cor}

\begin{proof}
We first show that all points of $\Sigma\setminus\{(0,0,F(0))\}$ are regular points. Indeed, suppose that $(x,y)\neq(0,0)$ is a characteristic point of $\Sigma$. 
Then,
\[ y = - 2x\frac{F'(|(x,y)|)}{|(x,y)|} \quad\text{and}\quad x = 2y\frac{F'(|(x,y)|)}{|(x,y)|}.\]
If $x=0$ or $y=0$, then both must be equal to 0, and this is a contradiction. 
However, if $x\neq 0$ and $y\neq 0$, then $y/x = x/y$ which implies that $x^2+y^2=0$,
and this is again a contradiction.
Thus the only characteristic point of $\Sigma$ is $(0,0,F(0))$.

Let $S$ be a compact subset of $\Sigma$. By Theorem \ref{thm:revolution}, Proposition \ref{thm:reg}, and the compactness of $S$, there exist open sets $D_1,\dots,D_k$ such that each $D_i$ bi-Lipschitz embeds into $\R^4$ and $S \subset \bigcup_{i=1}^k D_i$. By Lemma \ref{lem:welding}, we obtain a bi-Lipschitz embedding of $ \bigcup_{i=1}^k D_i$ (and thus of $S$) into $\R^{5k-1}$.
\end{proof}

%We note here another corollary:
%\begin{cor} The metric space $(\S^2,d_{\bH})$ bi-Lipschitz embeds into some Euclidean space $\R^n$. \end{cor}
%Later, however, we will construct a bi-Lipschitz embedding of this space into $\mathbb{R}^4$, so we save the proof for then (see Proposition~\ref{prop:sphereemb}).

We now prove Theorem~\ref{thm:revolution}.
The idea of the proof is as follows.
In Lemma~\ref{lem:BLrev},
we construct a bi-Lipschitz homeomorpism of a disc centered at the origin in $\mathbb{R}^2$
such that any radial segment in the disc is mapped to a curve which is the projection of a horizontal curve in $\Sigma$.
Since $\Sigma$ is a surface of revolution,
rotations of this horizontal curve foliate the surface.
We then complete the proof of Theorem~\ref{thm:revolution}
by re-parameterizing $\Sigma$ in terms of the horizontal curve
(which embeds into $\mathbb{R}$)
and the snowflaked unit circle
(which embeds into $\mathbb{R}^3$).

%\begin{lem}
%There exists $T>0$ with the following property. For each $s\in\S^1$ there exists a curve $\g_s:[0,T] \to \R$ with $\g_s(t) = (x_s(t),y_s(t))$ such that 
%\begin{enumerate}
%\item $\g_s(0) = (0,0)$ and $\g_s'(0) = s$,
%\item for all $t\in[0,T]$, $|\g_s'(t)| = 1$,
%\item the curve $(x_s(t),y_s(t),F(x_s(t)^2+y_s(t)^2))$ is horizontal.
%\end{enumerate}
%\end{lem}

\begin{lem}\label{lem:BLrev}
Let $F:[0,+\infty) \to \R$ be as in Theorem \ref{thm:revolution}. There exist $T>0$, $L=L(M)>1$, and an $L$-bi-Lipschitz map $G:B^2(0,T) \to B^2(0,T)$ with $G(0)=0$, such that for each $s\in\S^1$, the curve
\[ t \mapsto (G(ts), F(|G(ts)|))\]
is horizontal.
Moreover,
$|G(st)|=|G(s't)|$ for any $s,s' \in \S^1$ and $t \in [0,T)$.
\end{lem}

\begin{proof}
%It suffices to construct $G(st)$ for $s=(1,0)$. Then, revolving the curve around the $z$-axis and applying the invariance of $d_{\mathbb{H}}$ under rotations about the $z$-axis, we obtain the lemma. 
For each $t\geq 0$ consider the point $w_t = (t,0,F(t))$. Then
\[ 
T_{w_t}\Sigma \cap H_{w_t} =  \left\{\left( x , 2\frac{F'(t)}{t}(x-t),F(t) + F'(t)(x-t)\right) : x\in\R\right\}.
\]
Consider now the vector field $V:\R^2\setminus\{0\} \to \R^2$ given by
\[ V(ts) = s + 2(F'(t)/t)\hat{s}\]
where $t>0$, $s,\hat{s}\in\S^1$ and $\hat{s}$ is the rotation of $s$ by an angle of $\pi / 2$.
Note that $V$ has been defined in such a way that, 
when any integral curve $\gamma$ of $V$ is lifted to a curve on $\Sigma$,
this lifted curve is horizontal.
That is,
the curve $t \mapsto (\gamma(t),F(\gamma(t)))$ is a horizontal curve.

Fix $T>0$ such that 
$|F'(t)|/t \leq 2M$ for every $t \in (0,T)$.
That is, 
\[ 
|V(p)| = \sqrt{1+ \left(2\frac{F'(|p|)}{|p|}\right)^2}\in [1, 5M]
\quad
\text{for all } 
p\in B^2(\textbf{0},T)\setminus\{\textbf{0}\}.
\]
Then, for all $p\in B^2(\textbf{0},T)\setminus\{\textbf{0}\}$ we have
\begin{equation}\label{eq:NT}
V(p)\cdot p = |p|  \geq (5M)^{-1}|p||V(p)|. 
\end{equation} 
%Since $V|\R^2\setminus\{0\}$ is Lipschitz, it follows that every integral curve $\g$ is $C^1$. We show that, in fact, every integral curve is a chord-arc curve, that is, for all $x,y \in \gamma$, the subarc $\g(x,y)$ of $\g$ that joins $x$ with $y$ satisfies
%\[ \ell(\g(x,y)) \leq C|x-y|.\]
From (\ref{eq:NT}) we have that if $\gamma$ is an integral curve of $V$, then 
\[ (|\gamma|^2)' = \gamma'\cdot \gamma = |\gamma| = (|\gamma|^2)^{1/2}\] 
which implies that $(|\gamma|)' = 1/2$. Therefore, given $0<t_1 \leq t_2$ in the domain of $\gamma$,
\begin{equation}\label{eq:growth}
 |\g(t_2)| - |\g(t_1)| = \frac12(t_2-t_1).
\end{equation}

We construct a particular integral curve in $B(\textbf{0},T)$ as follows. 
For each $i\in \N$ let $S_i$ be the circle $\partial B^2(\textbf{0},2^{1-i}T)$, and let $A_i$ be the closed annulus $A(\textbf{0};2^{-i}T,2^{1-i}T)$. 
By (\ref{eq:growth}), every integral curve of $V|_{A_i}$ is a $C^1$ curve that joins $S_i$ with $S_{i+1}$
defined on an interval of length $2^{1-i}T$.

For each $i\in\N$, let $\tilde{\g}^{(i)}$ be the integral curve of $V|_{A_i}$ joining $S_{i+1}$ to the point $(2^{1-i}T,0)$
defined on the interval $[2^{1-i}T,2^{2-i}T]$. 
The vector field $V$ is defined in such a way that rotating $\tilde{\g}^{(i)}$ produces all integral curves of $V|_{A_i}$. 
Appropriately rotating each $\tilde{\g}^{(i)}$, we obtain a curve $\g_0$ that joins $\textbf{0}$ with the point $(T,0)$ 
such that $\tilde{\g}|_{A_i}$ is an integral curve of $V|_{A_i}$. 
Note that, by a limiting argument, (\ref{eq:growth}) also holds if one or both of the $t_i$ is zero.
Hence the domain of $\g_0$ is $[0,2T]$.

For each $s\in\S^1$, 
let $\gamma_s(t):[0,2T] \to \mathbb{R}^2$ be the curve obtained by rotating $\g_0$ 
in such a way that $\gamma_s(2T) = Ts \in \partial B^2(\mathbf{0},T)$.
By the rotation invariance of $V$, the integral curves of $V$ are exactly the curves $\{\g_s(t):s\in\S^1\}$ which join $\textbf{0}$ with $\partial B^2(\textbf{0},T)$. 
As discussed above, 
the lift of each $\g_s$ to $\Sigma$ is a horizontal curve since $\g_s'(t) =V(\g_s(t))$.

Define now $G:B^2(0,T) \to B^2(0,T)$ with $G(ts) = \g_s(2t)$. To show that $G$ is bi-Lipschitz, fix $s_1,s_2\in\S^1$ and $0\leq t_1\leq t_2 \leq T$. 
By the triangle inequality, the Mean Value Theorem, and (\ref{eq:growth}), 
for some $c \in [2t_1,2t_2]$,
\begin{align*} 
|G(t_1s_1)-G(t_2s_2)| &\leq |G(t_2s_1)-G(t_2s_2)| + |G(t_1s_1)-G(t_2s_1)|\\ 
&= |\g_{0}(2t_1)||s_1-s_2| + |\g_{s_1}(2t_1)-\g_{s_1}(2t_2)| \\
&= t_1|s_1-s_2| + |\g'_{s_1}(c)| |t_1-t_2| \\
&\leq t_1|s_1-s_2| + 5M |t_1-t_2| \\
&\simeq_M |t_1s_1 - t_2s_2|.
\end{align*}

For the other inequality we consider two cases.

\emph{Case 1.} Suppose that $10M|t_1-t_2| \geq t_1|s_1-s_2|$. By (\ref{eq:growth}),
\begin{align*} 
|G(t_1s_1)-G(t_2s_2)| 
\geq |G(t_2s_2)| - |G(t_1s_1)| 
= |\gamma_0(2t_2)| - |\gamma_0(2t_1)|
&= |t_1-t_2| 
\gtrsim_M |s_1t_1 - s_2t_2|.
\end{align*}

\emph{Case 2.} Suppose that $10M|t_1-t_2| < t_1|s_1-s_2|$. By (\ref{eq:growth}) and the Mean Value Theorem, for some $c\in[2t_1,2t_2]$,
\begin{align*}
|G(t_1s_1)-G(t_2s_2)| &\geq |G(t_2s_1) - G(t_2s_2)| - |G(t_2s_1)-G(t_1s_1)|\\
&=  |\g_{0}(2t_2)||s_1-s_2| - |\g_{s_1}'(c)||t_1-t_2| \\
&\geq t_1|s_1-s_2| - 5M|t_1-t_2|\\
&\geq \tfrac12 t_1|s_1-s_2|\\
&\gtrsim_M |s_1t_1 - s_2t_2|. \qedhere 
\end{align*}
\end{proof}

 We can now prove Theorem \ref{thm:revolution}. 

\begin{proof}[{Proof of Theorem \ref{thm:revolution}}]
For a point $w\in\R^3$, we denote by $\pi(w)$ the projection of $w$ on $\{z=0\}$. Since $d_{\bH}$ is invariant under vertical translations, we may assume that $F(0)=0$. Define 
\[ \Sigma_0 = \{w\in\Sigma : \pi(w) \in B^2(\textbf{0},T)\}\]
and choose $\e>0$ such that $B^3(\textbf{0},\e)\cap\Sigma \subset \Sigma_0$. Define now the homeomorphism 
\[ \Psi :  \{z=0\}\cap B^3(\textbf{0},T) \to \Sigma_0 \quad\text{as}\quad \Psi(st,0) = (G(st), F(|G(st)|))\] 
where $s\in\S^1$ and $t\in[0,T]$. 
We show that, if $T$ is chosen small enough, then the map $\Psi$ is $L$-bi-Lipschitz for some $L=L(M)>1$. 

Fix $s_1,s_2\in\S^1$, $0< t_1\leq t_2\leq T$ and points $w_1 = (s_1t_1,0) $ and $w_2 = (s_2t_2,0)$. 
By Lemma \ref{lem:z=0} it suffices to show that
\[ d_{\mathbb{H}}(\Psi(w_1),\Psi(w_2)) \simeq_M (t_2-t_1) + t_1|s_1-s_2|^{1/2}.\]
For the rest of the proof denote $w_1' = \Psi(w_1)$ and $w_2' = \Psi(w_2)$. 
We write $G(st) = (x(st),y(st))$. 
Recall that $G(st)=G(s't)$ lie on the same circle centered at the origin.
Therefore, 
since $d_{\bH}$ and $G(st)$ are both invariant under rotations, we may assume that $x(s_1t_1)=x(s_2t_1)$. Then
\begin{align}
d_{\bH}(w_1',w_2') &\simeq |G(s_1t_1) - G(s_2t_2)|\notag\\
&\quad + \Big| F(|G(s_1t_1)|) - F(|G(s_2t_2)|) + \tfrac12 x(s_1t_1)y(s_2t_2) - \tfrac12 x(s_2t_2)y(s_1t_1)\Big|^{1/2}\notag\\
%&\simeq_M  |s_1t_1 - s_2t_2| +  | F(|G(s_2t_1)|) - F(|G(s_2t_2)|) + \frac12 %x(s_2t_1)y(s_2t_2) - \frac12 x(s_2t_2)y(s_2t_1)\\
%& \qquad + \frac12 x(s_2t_2)(y(s_2t_1) - y(s_1t_1)) |^{1/2}\\
&\simeq |t_1-t_2| + t_1|s_1-s_2| \label{eq:est1} \\
&\quad +  \Big| F(|G(s_2t_1)|) - F(|G(s_2t_2)|) + \tfrac12 x(s_2t_1)y(s_2t_2) - \tfrac12 x(s_2t_2)y(s_2t_1)\notag\\
& \qquad + \tfrac12 x(s_2t_2)(y(s_2t_1) - y(s_1t_1)) \Big|^{1/2}.\notag
\end{align}

%Fix $s\in\S^1$, we write below $x'(st) = \partial_t x(st)$. 
Since the points $(G(s_2t_1), F(|G(s_2t_1)|))$ and $(G(s_2t_2), F(|G(s_2t_2)|))$ 
lie on the horizontal curve 
$t \mapsto (G(s_2t),F(|G(s_2t)|))$,  
Lemma \ref{lem:horiz} gives a constant $C_1 = C_1(M) > 0$ such that
\begin{align}
\left| F(|G(s_2t_1)|) - F(|G(s_2t_2)|) + \tfrac12 x(s_2t_1)y(s_2t_2) - \tfrac12 x(s_2t_2)y(s_2t_1) \right|^{1/2} &\lesssim \|\partial_t G(s_2t)\|_{\infty} |t_1-t_2| \notag \\
&= \|V(st)\|_{\infty}|t_1-t_2| \notag \\
&\label{eq:est2} \leq C_1 |t_1-t_2|
\end{align}

Set $\e = \min\{2^{-3/2}, (2C_1)^{-1},(12L^2)^{-1}\}$
where $L$ is the bi-Lipschitz constant of $G$.
(Note that $\e$ depends only on $M$.)
Also, we may choose $\d>0$ sufficiently small so that $|x(s_2t_1)| > |y(s_2t_1)|$ whenever $|s_1-s_2| < \d$. 
Indeed, we assumed earlier that $x(s_1t_1) = x(s_2t_1)$ so that the points $G(s_1t_1)$ and $G(s_2t_1)$ are reflections of one another across the $x$-axis.
Since $G$ is $L(M)$-bi-Lipschitz, such a $\d$ can be chosen depending only on $M$. 
Therefore, $|x(s_2t_1)| \geq \tfrac12 |G(s_2t_1)|$ when $|s_1-s_2| < \d$.
We consider the following cases.

\emph{Case 1:} $|s_1-s_2| \geq \d$. By Lemma \ref{lem:z=0}, $d_{\mathbb{H}}(w_1,w_2) \simeq_{\d} t_2$. On the other hand,
\begin{align*}
d_{\mathbb{H}}(w_1',w_2') \geq
 %|\pi(w_1') - \pi(w_2')|=
 |G(s_1t_1) - G(s_2t_2)| \gtrsim_M |s_1t_1-s_2t_2| &\simeq |t_1-t_2| + t_1|s_1-s_2| \simeq_{\d} t_2.
\end{align*}
Moreover, 
%applying the Mean Value Theorem, 
there exist $0 < \xi_1 < |G(s_1t_1)|$ and $0 < \xi_2 < |G(s_2t_2)|$ such that 
\begin{align*}
d_{\mathbb{H}}(w_1',w_2') &\leq d_{\mathbb{H}}(w_1',0) +d_{\mathbb{H}}(w_2',0)\\
&=|G(s_1t_1)| + |G(s_2t_2)| + |F(|G(s_1t_1)|) - F(0)|^{1/2}+ |F(|G(s_2t_2)|) - F(0)|^{1/2}\\
&\lesssim_M (t_1+t_2) + |F'(\xi_1)|^{1/2}|G(s_1t_1)|^{1/2} + |F'(\xi_2)|^{1/2}|G(s_2t_2)|^{1/2}.
\end{align*}
Assuming now that $T$ is sufficiently small we have
\begin{align*}
d_{\mathbb{H}}(w_1',w_2') &\lesssim_M t_2 + M^{1/2}(|\xi_1|^{1/2}|G(s_1t_1)|^{1/2} + |\xi_2|^{1/2}|G(s_2t_2)|^{1/2})\\
&\lesssim_M t_2 + (|G(s_1t_1)| + |G(s_2t_2)|)\\
&\lesssim_M t_2.
\end{align*}
Therefore, $d_{\mathbb{H}}(w_1',w_2') \simeq_{M,\d} t_2 \simeq_{M,\d} |w_1-w_2|$.

\emph{Case 2:} $|s_1-s_2| < \d$ and $|t_1-t_2| \geq \e t_1^{1/2}t_2^{1/2}|s_1-s_2|^{1/2}$. Note that 
\begin{align*}
|x(s_2t_2)(y(s_2t_1) - y(s_1t_1)) |^{1/2} 
\lesssim_M |x(s_2t_2)|^{1/2}t_1^{1/2}|s_1-s_2|^{1/2} 
&\lesssim_M t_1^{1/2}t_2^{1/2} |s_1-s_2|^{1/2} \lesssim_M |t_1 - t_2|.
\end{align*}
Hence by (\ref{eq:est1}), (\ref{eq:est2}), and Lemma~\ref{lem:z=0}
$$
d_{\mathbb{H}}(w_1',w_2') 
\lesssim_M 
|t_1 - t_2| + t_1|s_1-s_2| 
\lesssim
|t_1 - t_2| + t_1|s_1-s_2|^{1/2}
\lesssim
d_{\mathbb{H}}(w_1,w_2).
$$
Moreover,
\begin{align*}
d_{\mathbb{H}}(w_1,w_2)
\lesssim 
|t_1 - t_2| + t_1|s_1-s_2|^{1/2}
\leq
|t_1 - t_2| + t_1^{1/2}t_2^{1/2}|s_1-s_2|^{1/2}
&\lesssim_M
|t_1 - t_2| \lesssim_M
d_{\mathbb{H}}(w_1',w_2').
\end{align*}

\emph{Case 3:} $|s_1-s_2| < \d$ and $|t_1-t_2| < \e t_1^{1/2}t_2^{1/2}|s_1-s_2|^{1/2}$.  By (\ref{eq:est1}), (\ref{eq:est2}), and the choice of $\e$,
\begin{align*}
d_{\bH}(w_1',w_2') &\lesssim_M |t_1-t_2| + t_1|s_1-s_2| + |x(s_2t_2)|^{1/2}|y(s_2t_1)-y(s_1t_1)|^{1/2} + |t_1-t_2|\\ 
&\lesssim |t_1-t_2| + t_1|s_1-s_2| + |x(s_1t_2)|^{1/2}|y(s_2t_1)-y(s_1t_1)|^{1/2} \\
&\hspace{1.5in} + |x(s_2t_2) - x(s_1t_2)|^{1/2}|y(s_2t_1)-y(s_1t_1)|^{1/2}\\
&\lesssim_M |t_1-t_2| + t_1^{1/2}t_2^{1/2}|s_1-s_2| + t_1^{1/2}t_2^{1/2}|s_1-s_2|^{1/2} + t_1^{1/2}t_2^{1/2}|s_1-s_2|\\
&\lesssim |t_1-t_2| + t_1^{1/2}t_2^{1/2}|s_1-s_2|^{1/2}
\end{align*}
while
\begin{align*}
d_{\bH}(w_1,w_2)
&\gtrsim |t_1-t_2| + t_1|s_1-s_2|^{1/2} \\
&= |t_1-t_2| + t_2|s_1-s_2|^{1/2} - (t_2 - t_1)|s_1-s_2|^{1/2}\\
&> |t_1-t_2| + t_1^{1/2}t_2^{1/2}|s_1-s_2|^{1/2} - \sqrt{2} \e t_1^{1/2}t_2^{1/2}|s_1-s_2|^{1/2}\\
&\geq |t_1-t_2| + t_1^{1/2}t_2^{1/2}|s_1-s_2|^{1/2}.
\end{align*}

For the other direction, note that,
$|s_1-s_2| < \d$ gives $|x(s_2t_1)| \geq \tfrac12 |G(s_2t_1)|$. Thus,
\begin{align*}
|x(s_2t_2)| 
\geq |x(s_2t_1)| - |x(s_2t_2) - x(s_2t_1)| 
&\geq \tfrac12 |G(s_2t_1)| - |G(s_2t_2) - G(s_2t_1)| \\
&\geq \tfrac12 |G(s_2t_2)| - \tfrac32 |G(s_2t_2) - G(s_2t_1)|\\
&\geq \tfrac{1}{2L} t_2 - \tfrac{3L}{2}|t_1 - t_2|\\
&> \tfrac{1}{4L} t_2
\end{align*}
since $\e \leq 1/(12L^2)$. Therefore,
\begin{align*}
d_{\bH}(w_1',w_2') 
&\gtrsim_M |t_1-t_2| + \tfrac{1}{2}|x(s_2t_2)|^{1/2} |y(s_2t_1)-y(s_1t_1)|^{1/2} - C_1|t_1-t_2|\\
&\gtrsim_M |t_1-t_2| + t_1^{1/2}t_2^{1/2} |s_1 -s_2|^{1/2} - C_1|t_1-t_2|\\
&> |t_1-t_2| + t_1^{1/2}t_2^{1/2} |s_1 -s_2|^{1/2} - C_1\e t_1^{1/2}t_2^{1/2}|s_1-s_2|^{1/2}\\
&\geq |t_1-t_2| + t_1^{1/2}t_2^{1/2} |s_1 -s_2|^{1/2}\\
&\geq |t_1-t_2| + t_1 |s_1 -s_2|^{1/2} \\
&\gtrsim d_{\bH}(w_1,w_2)
\end{align*}
and the proof is complete. 
\end{proof}

\subsection{Kor\'{a}nyi spheres}\label{sec:spheres}

Following the ideas in \textsection\ref{sec:revolution}, we show directly that the Kor\'{a}nyi sphere 
\[ \mathbb{S}_K = \{ p \in \mathbb{H} \, : \, d_{K}(p,0) = 1\}\]
admits a bi-Lipschitz embedding into $\mathbb{R}^4$.
(Note that this is the sphere in the metric $d_K$ rather than $d_{\bH}$.)
Since the metric is left invariant on $\mathbb{H}$ and commutes with dilations, 
this shows that any Kor\'{a}nyi sphere admits a bi-Lipschitz embedding into $\mathbb{R}^4$.

\begin{prop}
\label{prop:ksphereemb}
There exists a bi-Lipschitz embedding of $(\S_K,d_{\bH})$ into $\R^4$.
\end{prop}

Since the proof is similar to that of Theorem \ref{thm:revolution}, we only sketch the steps and leave the details to the reader.

\begin{proof}
Given $s\in \S^1$ and $t\in[-1,1]$, we define 
\[ \ell_s(t) := (s(1-|t|),t) \in \R^3 \qquad\text{and}\qquad \mathcal{C} = \{\ell_s(t) : t\in[-1,1], s\in\S^1\}.\]
The arc $\ell_s([-1,1])$ is the \emph{longitude of $\mathcal{C}$ in the direction of $s$} with $\ell_s(-1)$ (resp. $\ell_s(1)$) being the south pole $\mathcal{S} = (0,0,-1)$ (resp. north pole $\mathcal{N} = (0,0,1)$) for all $s\in \S^1$. 
Given $\e\in (0,1)$, define the sets
\[ \mathcal{N}(\e) = \{\ell_s(t) : s\in\S^1, 1-\e\leq t \leq 1\} \quad\text{and}\quad \mathcal{S}(\e) = \{\ell_s(t) : s\in\S^1, -1 \leq t \leq -1+\e\}.\]
Define also the two poles of $\mathbb{S}_K$, $\mathcal{N}^* = (0,0,1/4)$ and $\mathcal{S}^* = (0,0,-1/4)$.

Working as in Lemma \ref{lem:BLrev}, we can show that there exists a bi-Lipschitz homeomorphism $G:\mathcal{C} \to \S_K$ such that $G(\mathcal{S}) = \mathcal{S}^*$, $G(\mathcal{N}) = \mathcal{N}^*$ and every longitude $\ell_s([-1,1])$ is mapped to a horizontal curve. In fact every curve $t \mapsto G(\ell_s(t))$ is an integral curve of the vector field
\[ V(x,y,z) = \frac{1}{\sqrt{4(x^2+y^2)+(x^2+y^2)^4}} \left(-8xz-2y(x^2+y^2),-8yz+2x(x^2+y^2),(x^2+y^2)^2\right) \]
which lies in the intersection of $T_p \S_K$ and $H_p$ at any $p = (x,y,z) \in \S_K$.

Define now $\Psi : (\S_K,d_{\bH}) \to \R^4$ by
\[ \Psi(G(s(1-|t|),t) = (t, (1-|t|)\phi(s))\]
where $\phi$ is the bi-Lipschitz embedding of $(\S^1,|\cdot|^{1/2})$ into $\R^3$. 
By \textsection\ref{sec:z=0} and Theorem \ref{thm:revolution}, we know that $\Psi$ is bi-Lipschitz on a neighborhood of $\mathcal{N}$ and on a neighborhood of $\mathcal{S}$. Moreover, for all $w \in \S_K \setminus\{\mathcal{N}^*,\mathcal{S}^*\}$, 
we can use the result of Lemma~\ref{lem:surfaces} to see that 
there exists a neighborhood of $w$ on $\S_K$ on which $\Psi$ is bi-Lipschitz
with a universal constant. 
By the compactness of $\S_K$, there exist $\e>0$ and $L>1$ such that $\Psi$ is $L$-bi-Lipschitz on each $B(p,\e)\cap\S_K$ with $p\in\S_K$. 
On the other hand, by Lemma \ref{lem:comparison}, if $w,w' \in \S_K$ with $|w-w'| \geq \e$, then
\[ d_{\bH}(w,w') \simeq_{\e} 1 \simeq_{\e} |\Psi(w)-\Psi(w')|. \qedhere\]
\end{proof}

We may follow very similar arguments to prove the following result.
\begin{prop}
\label{prop:sphereemb}
There exists a bi-Lipschitz embedding of $(\S^2,d_{\bH})$ into $\R^4$.
\end{prop}
As with Proposition \ref{prop:ksphereemb}, one needs to construct a bi-Lipschitz mapping $G:\S^2 \to \S^2$ such that every longitude $\ell_s([-1,1])$ is mapped to a horizontal curve. Here, the vector field $V$ is defined as 
\[ 
V(x,y,z) = \frac1{\sqrt{5(x^2+y^2)}}(xz + 2y, yz-2x, -y^2-x^2).
\]
The rest of the proof follows in the same way as above.

\subsection{The surface $z=\frac12xy$}\label{sec:1/2xy}
The manifolds considered in \textsection\ref{sec:revolution} and in \textsection\ref{sec:spheres} are manifolds that contain at most two characteristic points. Now we consider the surface $z = xy/2$ which has infinitely many characteristic points,
namely all points along the $x$-axis. 
Solving the ODE for $dx/dy$ in \eqref{eq:ODE1}, we obtain the horizontal curves
\[ 
\left\{\left(c,t,\tfrac{1}{2}ct \right) \, : \, t \in \mathbb{R} \right\} \text{ for } c \in \mathbb{R} 
\quad \text{and} \quad 
\{(t,0,0) \, : \, t \in \mathbb{R}\}.
\]
%The situation is completely different than the two above. 
%Here, all points are regular except for the line $\{(t,0,0)\}$ which intersects all horizontal curves. 
On the other hand, for any $a\in\R$ and $c\in\R$, the curve
\[ \{(t,c,ct/2) : t\in [a,a+c]\} \]
is a snowflake; that is, it is bi-Lipschitz homeomorphic to a rescaled copy of $\Phi([0,1])$ where $\Phi : \R \to \R^3$ is the $\frac12$-snowflaking map from \textsection\ref{sec:prelim}. The construction of the following embedding resembles a general construction by Seo \cite{Seo}.

\begin{prop}\label{prop:z=1/2xy}
The metric space $(\{z=\frac12xy\}, d_{\mathbb{H}})$ bi-Lipschitz embeds in $\R^{19}$.
\end{prop}

\begin{proof}
 For each $n,m\in\Z$ let 
\begin{align*}
Q_{n,m}^{\pm} &= \{(x,\pm y,\pm\tfrac12xy) :  m2^{-n} \leq x \leq (m+1)2^{-n} , 2^{-n} \leq y \leq 2^{1-n}\}
%Q_{n,m}^- &= \{(x,-y,-\tfrac12xy) : m2^{-n} \leq x \leq (m+1)2^{-n} ,  2^{-n} \leq y \leq 2^{1-n}\}.
\end{align*}

Firstly, we show that $(Q_{0,0}^+,d_{\mathbb{H}})$ bi-Lipschitz embeds in $[1,2] \times [0,1]^3 \subset \mathbb{R}^4$. Indeed, consider the map
\[ g: (Q_{0,0}^+,d_{\mathbb{H}}) \to \R^4 \quad\text{ with }\quad g(x,y,\tfrac12xy) = (y,\Phi(x)).\]
Then, for $w=(x,y,\frac12xy) \in  Q_{0,0}^+$ and $w' = (x',y',\frac12x'y') \in Q_{0,0}^+$, we have
\begin{align*}
d_{\mathbb{H}}(w,w') &= |x-x'| + |y-y'| + 2^{-1/2}(y+y')^{1/2} |x-x'|^{1/2} \\
%&\simeq |x-x'| + |y-y'| + |x-x'|^{1/2} \\
&\simeq |y-y'| + |x-x'|^{1/2}\\
&\simeq |g(w)-g(w')|.
\end{align*}

Secondly, we remark that the maps $\zeta_{n,m}^{\pm} : (Q_{0,0}^+,d_{\mathbb{H}}) \to (Q_{n,m}^{\pm},d_{\mathbb{H}})$ given by
\[ \zeta_{n,m}^{\pm}(x,y,\tfrac12xy) = (2^{-n}(x+m), \pm2^{-n}y, \pm\tfrac{1}{2}4^{-n}(x+m)y)\]
are $2^{-n}$-similarities
i.e. $d_{\mathbb{H}}(w,w') = 2^{-n} d_{\mathbb{H}}(\zeta_{n,m}^{\pm}(w),\zeta_{n,m}^{\pm}(w'))$. 
%A $2^{-n}$-similarity map $\zeta_{n,m}^{-} :  (Q_{0,0}^+,d_{\mathbb{H}}) \to (Q_{n,m}^-,d_{\mathbb{H}})$ can be defined analogously. 
Let also $\sigma_{n,m}^{\pm} : \R^4 \to \R^4$ be the similarities given by
\[ \sigma_{n,m}^{\pm}(x,y,z,t) = (\pm 2^{-n}x, 2^{-n} (y+m), 2^{-n}z, 2^{-n}t).\]

Thirdly, we decompose $(\{z=\frac12xy\}, d_{\mathbb{H}})$ into 4 pieces, 
we embed each piece bi-Lipschitz into $\R^4$, 
and we glue these 4 embeddings via Lemma~\ref{lem:welding}. %Set $\mathscr{Q}= \{Q_{n,m} : n,m\in\Z\}$. 
%We claim that there exist families $\mathscr{Q}_1, \dots, \mathscr{Q}_N$ with the following properties
%\begin{enumerate}
%\item $\mathscr{Q} = \bigcup_{i=1}^N \mathscr{Q}_i$;
%\item for all $i\in\{1,\dots,N\}$ and all $Q,Q' \in \mathscr{Q}_i$ we have
%\[ \dist(Q,Q') \gtrsim \max\{\diam{}\}\]
%\end{enumerate}
Define the following four families of squares
\begin{align*}
&\mathscr{Q}_{1} = \{ Q^{\pm}_{n,m} : n\text{ is even and }m\text{ is even}\} \quad &&\mathscr{Q}_{2} = \{ Q^{\pm}_{n,m} : n\text{ is even and }m\text{ is odd}\}\\
&\mathscr{Q}_{3} = \{ Q^{\pm}_{n,m} : n\text{ is odd and }m\text{ is even}\} \quad &&\mathscr{Q}_{4} = \{ Q^{\pm}_{n,m} : n\text{ is odd and }m\text{ is odd}\}.
 \end{align*}
%We construct bi-Lipschitz embeddings $G_i : (\bigcup\mathscr{Q}_{i},d_{\mathbb{H}}) \to \R^4$. 
Fix an $i\in\{1,2,3,4\}$ and define $G_i : (\{y=z=0\}\cup\bigcup\mathscr{Q}_{i},d_{\mathbb{H}}) \to \R^4$ by 
\[ G_i|_{Q_{n,m}^{\pm}} = \sigma^{\pm}_{n,m} \circ g \circ (\zeta_{n,m}^{\pm})^{-1}|_{Q_{n,m}^{\pm}} \quad\text{and}\quad G_i(x,0,0) = (0,x,0,0).\]
By the first two steps, for each $Q \in \mathscr{Q}_i$, the map $G_i|_Q$ is bi-Lipschitz with bi-Lipschitz constant independent of the size of $Q$. It remains to show that each $G_i$ is bi-Lipschitz. Assuming this is true, $(\{z=\frac12xy\} , d_{\mathbb{H}})$ bi-Lipschitz embeds into $\R^{19}$ by Lemma \ref{lem:welding}.

Fix $Q, Q' \in \mathscr{Q}_i$ with $Q\neq Q'$ and fix $w=(x,y,\frac12xy)\in Q$ and $w' = (x',y',\frac12x'y') \in Q'$. (The case that one of them is on the $x$-axis follows by a limit argument). Recall that 
\[ d_{\mathbb{H}}(w,w') \simeq |x-x'| + |y-y'| + |y+y'|^{1/2}|x-x'|^{1/2}.\]
We claim that $d_{\mathbb{H}}(w,w') \simeq |x-x'| + |y-y'|$. To see that, consider the following two cases.

\emph{Case 1.} Suppose that $Q= Q_{n,m}^+$ and $Q' = Q_{n,m'}^+$ for $m \neq m'$. Then, 
\[ %|y-y'| \leq 
|y+y'| \lesssim 2^{-n} \leq |m-m'| 2^{-n} \lesssim |x-x'|.\]
Hence, $|y+y'|^{1/2}|x-x'|^{1/2} \lesssim |x-x'|$. The same conclusion holds if $Q= Q_{n,m}^-$ and $Q'= Q_{n,m'}^-$.

\emph{Case 2.} Suppose that Case 1 does not hold. Then $Q= Q_{n,m}^\pm$ and $Q' = Q_{n',m'}^\pm$ for $n \neq n'$, and so
$|y-y'| \simeq |y+y'|$. The claim follows.

With similar reasoning, we can show that, for any $Q,Q' \in \mathscr{Q}_i$ with $Q\neq Q'$ and for any $w=(x,y,\frac12xy)\in Q$ and $w' = (x',y',\frac12x'y') \in Q'$, we have
\begin{equation}
\label{Gi} 
|G_i(w) - G_i(w')| \simeq |x-x'| + |y-y'|.
\end{equation}
Indeed, notice that $(g \circ (\zeta_{n,m}^{\pm})^{-1})(Q) \subset [1,2]\times [0,1]^3$
for any $Q = Q_{n,m}^\pm \in \mathscr{Q}_i$,
and thus 
$$
G_i(Q) \subset [\pm 2^{-n}, \pm 2^{1-n}] \times [m2^{-n},(m+1)2^{-n}] \times [0,2^{-n}] \times [0,2^{-n}].
$$
According to the definition of $\mathscr{Q}_i$ and arguments as above, then, \eqref{Gi} holds
for any $Q,Q' \in \mathscr{Q}_i$ with $Q\neq Q'$.
That concludes the proof of Proposition \ref{prop:z=1/2xy}.
\end{proof}

\section{Nonembedabbility of Heisenberg porous sets}\label{sec:porous}

The focus of this section is the proof of Theorem \ref{th:bad-porous}. For the rest of \textsection\ref{sec:porous}, we set $\pi : \H \to \R^2$ to be the 1-Lipschitz abelianization map. Given $\alpha > 0$, we define a $\alpha$-cone in $\H$ as the set
\begin{equation}
\label{conedfn}
  C_\alpha := \{(x,y,z) : |z| \leq \alpha (x^2 + y^2)\}.
\end{equation}
It is easily seen that $C_\alpha$ is dilation and rotation invariant and $C_\alpha \subseteq C_\beta$ when $\alpha \leq \beta$. 
In \textsection\ref{sec:laakso} we define the set $X$
and explain why there is no bi-Lipschitz embedding of $X$ in $\ell^2$, 
and in \textsection\ref{sec:laaksoproof} we show that $X$ is indeed porous.

\subsection{Laakso graphs}\label{sec:laakso}
Let $\{G_n\}_{n=1}^\infty$ denote the sequence of Laakso graphs defined in Section 2.2 of \cite{li} and let $f_n : G_n \to \H$ denote the embedding defined starting from the bottom of p. 1625 (there it was just denoted $f$ by abuse of notation).  
(Note in particular that each $G_n$ consists only of the vertices in the graph and does not include the edges.)
Recall that $f_n$ is specified with a sequence of angles $\theta_j = \left( \sqrt{M + j} \log (M + j) \right)^{-1}$ for some $M$ sufficiently large to be fixed.  Theorem 1.2 of \cite{li} showed that $f_n$ on $G_n$ has distortion $O((\log |G_n|)^{1/4} \sqrt{\log \log |G_n|})$.

\begin{figure}[ht]
\centering
\includegraphics[width=6in]{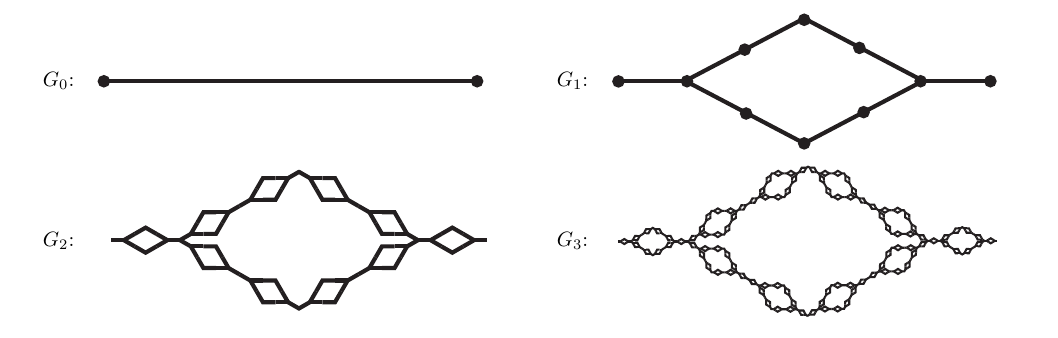}
\caption{The construction of $\{G_n\}_{n=1}^\infty$.}
\end{figure}

One of the main goals of \cite{li} was to use $f_n$ to show that the image $f_n(G_n)$ cannot uniformly bi-Lipschitz embed into Markov $p$-convex metric spaces when $p < 4$.  This is explicitly done in the proof of Corollary 1.4 of \cite{li} (although one needs to remove the intermediary $\H(\Z)$ that appears there).  Since Hilbert space is Markov 2-convex, we get that $f_n(G_n)$ do not embed into $\ell_2$ with uniform bi-Lipschitz distortion.

We now require that the diameter of each $G_n$ is 1.  Thus, the length of edges in $G_n$ become $6^{-n}$.    
This allows us to define a limiting compact object $G_\infty$ also of diameter 1.
It is not hard to show that $G_\infty$ contains each $G_n$ isometrically.

One also has that $G_\infty$ is self-similar and is composed as a union of 10 scaled copies of $G_\infty$.  
Call these level 1 copies.  We use $H_1$ to denote any one of these copies, and they are all isometric.  
(Note that $H_1$ could refer to any level 1 copy of $G_\infty$.)  
Each of these in turn are composed of 10 scaled copies of $G_\infty$.  
Call these level 2 copies, and use $H_2$ to denote each of these copies.  
We can continue and identify level $n$ copies for all $n \in \N$ each denoted $H_n$.  
If a copy $H_i$ does not contain a terminal point of some $H_{i-1}$ we say it is forked.  
Note that each $H_i$ contains 8 forked and 2 non-forked copies of $H_{i+1}$.

Each $H_i$ contains a unique scaled
copy of $G_k$ that is maximal in diameter.  
We call this $H_{i,k}$.  Again, $H_{i,k}$ could refer to any one of many isometric objects.

We can define a Lipschitz map $f : G_\infty \to \H$ so that the restriction of $f$ on each embedded copy of $G_n$ in $G_\infty$ is a re-scaled (and possibly translated and rotated) copy of $f_n(G_n)$.  Thus, $f(G_\infty)$ is compact and cannot embed into $\ell_2$ as it contains copies of $f_n(G_n)$.  
Define $X = f(G_\infty)$.

\subsection{Proof of Theorem \ref{th:bad-porous}}\label{sec:laaksoproof}
From the discussion above, Theorem \ref{th:bad-porous} immediately follows from the following proposition.

\begin{prop} \label{p:porous}
  The set $X = f(G_\infty)$ is porous in $\H$.
\end{prop}

For the rest of \textsection\ref{sec:laaksoproof}, given a subset $A \subset G_\infty$, we let $\tilde{A} := \pi(f(A)) \subset \R^2$,
and if $x \in G_\infty$, set $\tilde{x} := \pi(f(x))$.
We require two lemmas from \cite{li}.  The first one follows from Lemma 2.6 of \cite{li}.

\begin{lem} \label{l:terminals}
  For $M$ sufficiently large, if $s,t$ are terminal points of a level $n$ copy $H_n$ of $G_\infty$, 
  then $\frac{6^{-n}}{2} \leq |\pi(f(s)) - \pi(f(t))| \leq d_{\mathbb{H}}(f(s),f(t))$.
\end{lem}

The second is Lemma 2.7 of \cite{li}.

\begin{lem} \label{l:convex-hull}
 For any $i$, the set $\tilde{H}_i$ is contained in the closed convex hull of the corresponding $\tilde{H}_{i,1}$.
\end{lem}

\begin{rem} \label{r:triangle}
 Let $C$ and $C'$ be convex hulls of $\tilde{H}_{i+1,1}$ 
for copies of $H_{i+1}$ in $H_i$ that are in parallel and intersect at a fork point $a$. 
 Lemma \ref{l:convex-hull} says that there exists a cone based at $a$ of aperture at $\theta_i - \frac{\theta_{i+1}}{2} \geq \frac{\theta_i}{2}$ between $C$ and $C'$.
\end{rem}

\begin{cor} \label{l:H-bound}
  For any $H_n$, we have $\frac{6^{-n}}{2} \leq \diam_{d_{\bH}}{f(H_n)} \leq 6^{-n}$.
\end{cor}

\begin{proof}
  The upper bound comes from 1-Lipschitzness of $f$.  The lower bound comes from Lemma \ref{l:terminals}.
\end{proof}

\begin{prop} \label{p:flat}
  There exists $M>0$ sufficiently large so that the following holds.  Let $H_i \supset H_{i+1} \supset ... \supset H_k$ be a nested sequence of sub-graphs so that $H_{i+1}$ is forked in $H_i$.  Let $S \subseteq \{i+2,...,k\}$ be the indices $j$ so that $H_j$ is forked in $H_{j-1}$.  Suppose there exists another forked copy $H_{i+1}' \subset H_i$ that is in parallel with $H_{i+1}$ so that there are points $x \in H_k$ and $y \in H_{i+1}'$ for which $6^{-k} \leq |\tilde{x} - \tilde{y}| \leq 6^{-k + 1}$.  Then $\sum_{j \in S} \theta_j \leq 1/100$.
\end{prop}

\begin{proof}
  We suppose for contradiction that $\sum_{j \in S} \theta_j > \frac{1}{100}$.  In particular, $S$ is nonempty so let $\ell = \min S$ and $a,b,c \in \R^2$ be the points of $\tilde{H}_{i,1} \subset \tilde{H}_i$ with degree at least 3.  
  Then the distance from $\tilde{x}$ to $a,b,c$ is at least $6^{-\ell}/2$ by Lemma \ref{l:terminals}.  
  By Remark \ref{r:triangle}, the distance from $x$ to any point in $\tilde{H}_{i+1}'$ (which is in parallel with $\tilde{H}_{i+1}$)
  is at least $6^{-\ell}\theta_i/8$ and so $6^{-k+1} \geq 6^{-\ell}\theta_i/8$ by the assumption of the lemma.  
  This tells us that $k \leq \ell + 3 + \log \frac{1}{\theta_i}$ and so $|S| \leq 3 + \log \frac{1}{\theta_i} \leq \frac{1}{100}\sqrt{i + M}$
  for large enough $M$.  Thus,
  \begin{align*}
    \sum_{j \in S} \theta_j \leq \sum_{j=1}^{\sqrt{i+M}/100} \theta_{i+1+j} \leq \sum_{j=1}^{\sqrt{i+M}/100}\frac1{\log(M+i+j+1)\sqrt{M+i+j+1}} < \frac{1}{100},
  \end{align*}
  which is a contradiction.
\end{proof}

We also choose $M$ sufficiently large so that $100 \theta_1 < \frac{1}{100}$.

Fix some $x \in X$ and choose some $k \in \N$.  Let $y \in G_\infty$ be any point so that $f(y) = x$.  For each $i \in \N$, there exists some $H_i$ that contains $y$.  We call this subset $Y_i$.  If $y$ is a terminal point of some $H_i$, then the choice of $Y_i$ may not be unique. In this case, choose $Y_i$ arbitrarily while ensuring that, at each step, $Y_{i+1} \subset Y_i$.  We let $s_i$ and $t_i$ denote the source and sink of $Y_i$.

Let $(u,v) \in \S^1$ be a $90^\circ$ clockwise rotation of the direction from $\tilde{s}_k$ to $\tilde{t}_k$.  
Let 
\[ x' = f(s_k) \cdot (u6^{-k+10}, v6^{-k+10}, 0)\] 
and define the ball $B = B(x', 6^{-k})$.  
Our goal is to prove that $B \cap X = \emptyset$.  
It suffices to show that $B \cap (f(Y_i) \backslash f(Y_{i+1})) = \emptyset$ for every $i \in \N$.  
%\textcolor{red}{More detail?}
If $k \leq i + 5$, 
then it is obvious that this intersection is empty
by a simple triangle inequality argument and the fact that $f(Y_i) \subseteq B(f(s_k), 6^{-k+5})$.

We may now suppose $i < k - 5$. Let $S \subseteq \{i+2,...,k\}$ be the indices $j$ so that $Y_j$ is forked in $Y_{j-1}$.  Let $P : \R^2 \to \R$ denote the orthogonal projection onto the linear subspace spanned by $\tilde{s}_i$ and $\tilde{t}_i$ with the orientation that $P(\tilde{s}_i) < P(\tilde{t}_i)$.

\begin{lem}
  Suppose at least 100 of $\{Y_j\}_{j=i+2}^k$ are forked. Then
  \begin{align}
    P(\pi(B)) \subseteq [P(\tilde{s}_{i+1}) + 10 \cdot 6^{-k}, P(\tilde{t}_{i+1}) - 10 \cdot 6^{-k}]. \label{e:p-ball}
  \end{align}
\end{lem}

\begin{proof}
  Let $\ell \in \{i+2,...,k\}$ be the smallest integer so that $Y_\ell$ is forked.  Thus, in $Y_\ell$, there exists a non-forked $H_{\ell+1}$ separating $Y_k$ from $s_\ell$ and $t_\ell$
  (since at least one $Y_j$ is forked for $j \in \{ \ell + 2, \dots, k \}$).  
  As $\ell$ was minimal in $\{i+2,...,k\}$, $P(\tilde{s}_\ell) \geq P(\tilde{s}_{i+1})$ and we get that $P(\tilde{s}_k) - P(\tilde{s}_\ell) > 6^{-\ell}/2$.  
  As $k \geq \ell + 99$ and since the group multiplication acts in the first two coordinates simply as addition, this proves the lemma
  (arguing similarly for the right hand side of the interval).
\end{proof}

\begin{lem} \label{l:fork}
  If $B$ intersects $f(Y_i) \backslash f(Y_{i+1})$, then $Y_{i+1}$ must be forked and $B$ must intersect some other $f(H_{i+1})$ where $H_{i+1}$ is in parallel with $Y_{i+1}$.
\end{lem}

\begin{proof}
Since $B \cap f(Y_i) \neq \emptyset$,
the ball $B$ intersects some $f(H_{i+1})$ in $f(Y_i)$.
We denote this $H_{i+1}$ as $Y_{i+1}'$.
  Let $\theta$ denote the angle between the line segments connecting the terminals of $\tilde{Y}_i$ and $\tilde{Y}_k$.  First suppose that $\theta \leq 1/100$.  Then an elementary geometric argument in $\R^2$ using Lemma \ref{l:convex-hull} shows that $\pi(B)$ cannot intersect any $\tilde{H}_{i+1}$ where $H_{i+1}$ is in series with $Y_{i+1}$.  
  Thus $Y_{i+1}'$ must be in parallel with $Y_{i+1}$, and this proves the lemma in this case.

  Now suppose $\theta > 1/100$.  Then as $100\theta_1 \leq 1/100$, we must have that at least 100 of $\{Y_j\}_{j=i+2}^k$ are forked.  
  %\textcolor{red}{More detail?}
  The lemma then follows from the previous lemma.
  Indeed, $P(\pi(B)) \cap P(\tilde{Y}_{i+1}') \neq \emptyset$, and so $Y_{i+1}'$ cannot be in series with $Y_{i+1}$.
\end{proof}

Recall that $C_1$ is the 1-cone defined in \eqref{conedfn}.

\begin{lem} \label{l:cone}
  $6B \subseteq f(s_k) \cdot C_1$.
\end{lem}

\begin{proof}
  We may suppose $f(s_k) = (0,0,0)$.  Let $p \in B(0,6^{-k+1})$.  Then $(u6^{-k+10},v6^{-k+10},0) \cdot p$ has a $z$-coordinate of absolute value at most $6^{-2k+2} + 6^{-k+10} \cdot 6^{-k+1} \leq 6^{-2k+12}$.  On the other hand, the $x$ and $y$ coordinates have absolute value at least $6^{-k+9}$.  This easily proves the lemma (with drastic overestimation).
\end{proof}

We now show Proposition \ref{p:porous}.

\begin{proof}[{Proof of Proposition \ref{p:porous}}]
Suppose $B$ intersects $f(Y_i) \backslash f(Y_{i+1})$.  Let $Y_{i+1}'$ denote the $H_{i+1}$ that is in parallel with $Y_{i+1}$ such that $B \cap f(Y_{i+1}')$ is nonempty (from Lemma~\ref{l:fork}).  Let $p \in B \cap f(Y_i')$. In a similar way, for $j \in \{i+2,...,k\}$, let $Y_j'$ denote the $H_j$ that contain $p$ with $Y_j' \subseteq Y_{j-1}'$ and $s_k',t_k',\tilde{s}_k',\tilde{t}_k'$ be defined similarly.  Note that the inflated ball $6B$ must intersect $\tilde{s}_k'$.  Let $S'$ be the indices $j$ of $\{i+2,...,k\}$ where $Y_j'$ is forked in $Y_{j-1}'$.  By Proposition \ref{p:flat}, we have that both $\sum_{j \in S} \theta_j$ and $\sum_{j \in S'} \theta_j$ are less than 1/100.

Take a geodesic path from $s_k$ to $s_k'$.  
This is pushed into $\H$ via $f$ to a piecewise affine horizontal path $p_1p_2\cdots p_n$ so that $p_1 = f(s_k)$ and $p_n = f(s_k')$.  
As $\max\{\sum_{j \in S} \theta_j,\sum_{j \in S'} \theta_j\} \leq 1/100$ and the line from $\pi(p_1)$ to $\pi(p_n)$ is almost orthogonal to the line between the terminals of $\tilde{Y}_i$, %\textcolor{red}{(More detail?)}
the closed piecewise affine path $\pi(p_1)\cdots \pi(p_n) \pi(p_1)$ in $\R^2$ does not self intersect.  It follows from Remark \ref{r:triangle} that this path encloses an isosceles triangle of angle $\theta_i - \frac{\theta_{i+1}}{2}$ and base at least $\frac{|\pi(p_n) - \pi(p_1)|}{2}$ which has area
\begin{align*}
  \frac{1}{2} \frac{|\pi(p_1) - \pi(p_n)|}{2} \frac{|\pi(p_1) - \pi(p_n)|}{2 \tan\left(\theta_i - \frac{\theta_{i+1}}{2} \right)} \geq \frac{1}{16\theta_i} |\pi(p_1) - \pi(p_n)|^2.
\end{align*}
Thus, we see that $p_n \notin p_1 \cdot C_{1/16\theta_{i-1}}$.  As $\theta_i$ is decreasing, this means $p_n \notin p_1\cdot C_{1/16\theta_1}$.  However, from Lemma \ref{l:cone} we have that $p_n \in p_1 \cdot C_1$.  This is a contradiction, which means $B$ does not intersect $f(Y_i) \backslash f(Y_{i+1})$.
This proves that $X = f(G_\infty)$ is porous in $\H$.
\end{proof}

\bibliographystyle{alpha}

\bibliography{Heisenberg}
\end{document}